\documentclass[final]{siamltex}
\usepackage{graphicx,color,keyval,epstopdf}
\usepackage{amsgen,amssymb,amsfonts,amsbsy,amsmath,amscd,stmaryrd}
\usepackage{times}
\usepackage{pstricks}
\usepackage{pst-node,pst-text,pst-pdf,pst-3d}
\usepackage{subfigure}

\newcommand{\R}{\mathbb{R}}
\newcommand{\N}{\mathbb{N}}

\newcommand{\Z}{\mathbb{Z}}
\newcommand{\bfx}{\textbf{x}}
\newcommand{\calc}{\mathcal{C}}
\newcommand{\dsp}{\displaystyle}
\newcommand{\eps}{\varepsilon}
\newcommand{\abs}[1]{\lvert#1\rvert}
\newcommand{\norm}[1]{\lVert#1\rVert}
\newcommand{\Soniacommente}[1]{\textcolor{red}{}}
\newcommand{\lracco}[1]{\left\{#1\right\}}
\newcommand{\lrpar}[1]{\left(#1\right)}
\newcommand{\lrres}{\left.\right|}
 \newtheorem{remark}[theorem]{Remark}

\title{A Dirichlet-to-Neumann approach for the exact computation of guided modes in photonic crystal waveguides}
\begin{document}
\author{Sonia Fliss\footnotemark[2]}
\renewcommand{\thefootnote}{\fnsymbol{footnote}}
\footnotetext[2]{POEMS, (UMR 7231 CNRS/ENSTA/INRIA), \\
 ENSTA ParisTech, \\
828, boulevard des Mar\'echaux,
91762 Palaiseau Cedex \\
 \texttt{sonia.fliss@ensta-paristech.fr}}

\renewcommand{\thefootnote}{\arabic{footnote}}

\maketitle

\begin{abstract} 
This works deals with one dimensional infinite perturbation - namely line defects - in periodic media. In optics, such defects are created to construct an (open) waveguide that concentrates light. The existence and the computation of the eigenmodes is a crucial issue. This is related to a self-adjoint eigenvalue problem associated to a PDE in an unbounded domain (in the directions orthogonal to the line defect), which makes both the analysis and the computations more complex. Using a Dirichlet-to-Neumann (DtN) approach, we show that this problem is equivalent to one set on a small neighborhood of the defect. On contrary to existing methods, this one is exact but there is a price to be paid : the reduction of the problem leads to a nonlinear eigenvalue problem of a fixed point nature.
\end{abstract}


\pagestyle{myheadings}
\thispagestyle{plain}

\section{Introduction}

\noindent Periodic media play a major role in applications, in particular in optics for micro and nano-technology \cite{Joannopoulos:1995,Johnson:2002,Kuchment:2001,Sakoda:2001}.
From the point of view of applications, one of the main interesting features is that it can exist intervals of frequencies for which the propagative waves cannot exist in the media. This phenomenon is due to the fact that a wave on the media is multiply scattered by the periodic structure, which can lead, depending on the characteristics of the media and the frequency, to possibly destructive interferences. It seems then that an appropriate choice of the structure and the dielectric materials of the photonic crystal can create particular band gap and then, from a practical point of view, banish some monochromatic electromagnetic waves. Thus, the periodic media could be used to several potential applications such as in the realization of filters, antennas and more generally, components used in telecommunications. 
\\\\
Mathematically, this property is linked to the gap structure of the spectrum of the underlying differential operator appearing in the model. For a complete, mathematically oriented presentation, we refer the reader to \cite{Kuchment:2001,Kuchment:2004}. Even if the necessary conditions for the existence of band gaps are not known -except for the one dimensional case in \cite{Borg:1946} where authors state that the absence of gaps implies that the coefficients of the media are constant-, sufficient conditions exist. Figotin and Kuchment have given examples of high contrast medium for which at least one band gap exists and can be characterized \cite{Figotin:1996a,Figotin:1996b}. Using asymptotic arguments, Nazarov and co-workers have established that a small perturbation of an homogeneous waveguide can open a gap in the continuous spectrum of the operator \cite{Nazarov:2010,Cardone:2009}. Moreover, other band gap structures have been characterized through numerical approaches in \cite{Figotin:1997a}. 
 \\\\
Besides, Sommerfeld and Bethe in 1933 made the conjecture that the number of gaps has to be finite in 2D and 3D. This conjecture has been proven by Skriganov for Schr\"odinger operator in 2D \cite{Skriganov:1979, Popov:1981} and in 3D \cite{Skriganov:1983,Skriganov:1985} - the last article goes further: if the contrast is too small, gaps cannot exist and on suitable assumptions on the lattice, the conjecture holds for higher dimensions. Parnovski \cite{Parnovski:2008} has generalized this last result provided that the potential is sufficiently smooth and in \cite{Parnovski:2010} the conjecture was proven in full generality. Another proof can be found in the book of Yu Karpeshina \cite{Karpeshina:1997}. For Maxwell equations, it is much more complicated. Nevertheless, there is now evidence that gaps exists. Experimentally this was first observed by Yablonovitch et al \cite{Yablonovitch:1991}. There are also many numerical results for different structures (see \cite{Soukoulis:1993,Soukoulis:1996} for references). Recently, Vorobets in \cite{Vorobets:2011} has proven the conjecture for 2D periodic Maxwell operator (which corresponds to a 2D periodic photonic crystal which is homogeneous in the third direction) with separable dielectric function. Moreover, if the dielectric function is close enough to a constant, there is no gaps at all.
\\\\
These media can present perturbations or defects which are introduced in the media to change their properties. Thus, in optics, in order to produce lasers, fibers or waveguides in general, it is necessary to have authorized frequencies inside the forbidden intervals of frequencies. This property can be obtained, for example, introducing localized defect or line defect and corresponds, from a mathematical point of view, to isolated eigenvalues of finite multiplicity inside the gaps of the underlying differential operator appearing in the model.
\\\\
Figotin and Klein have proven rigorously that introducing a defect in a periodic structure, namely a perturbation of compact support, can create defect modes, which are eigenvectors associated to eigenvalues inside a gap \cite{Figotin:1997,Figotin:1998a,Figotin:1998b}. More precisely, these defect modes are stationary waves exponentially decreasing far from the perturbation, then they seems to "live" around the defect which explains their name. Using asymptotic arguments, Nazarov has studied in \cite{Nazarov:2011} sufficient conditions for existence of eigenvalues below the essential spectrum for elastic waveguide. Here again, it seems difficult to obtain necessary conditions for the existence of defect modes. 
\\\\
There are very few works in the same spirit in the mathematical literature for the case of line defect, that means a one dimensional perturbation of the periodic medium. A first natural question concerns the spectrum of the perturbed operator compared to the spectrum of the perfectly periodic operator. More precisely, if a gap is in the spectrum of the perfectly periodic operator, can a part of the spectrum of the perturbed one arise in this gap? Moreover, if a part of a spectrum arises, is it of absolutely continuous or singular type? Does it corresponds to \emph{guided modes}, meaning that the waves are confined to the guide (evanescent in the periodic media) and they are propagating along the line defect? Kuchment et Ong \cite{Kuchment:2003} have answered partially to these questions for the case of homogeneous line defect. Indeed, for any gap $(a,b)$ of the perfectly periodic operator, they found a sufficient condition on $a,\,b$, the characteristics of the line defect (its size and its coefficients) such that spectrum of the perturbed operator arises in the gap. Moreover, they  have shown that the associated eigenvector is confined to the guide. However, they have not concluded on the existence or not of bound states - this question is linked to the nature of the spectrum. Vu Hoang and Maria Radosz in \cite{Hoang:2012} have shown for waveguide in 2D periodic structures the absence of bound states. It seems then that if the spectrum arises in the gap, it corresponds to guided modes. Let us remark that for the particular case of a line periodic perturbation in a homogeneous media, Bonnet-BenDhia and Starling \cite{Bonnet:1994} have given necessary and sufficient conditions for the existence of guided modes.
\\\\
Finally, few works (\cite{Ammari:2004} and the present one) characterized precisely (respectively using Green's function or DtN operator) the guided waves if they exist, for general line defect. These particular modes can help for the determination of the spectrum of the perturbed operator. The advantage is that the study can be reduced to a band (bounded in the direction of the periodicity and infinite in the other direction). 
\\\\
From a numerical point of view, there exist only few methods. The most known is the \emph{Supercell method}. It consists in making computations in a bounded domain of large size with periodic boundary conditions, the resulting solution converging to the true solution when the size tends to infinity. The convergence for the computation of defect modes has been shown for 2D problems and compact perturbations in \cite{Soussi:2005} and generalized to 3D problems and to exponentially decreasing perturbations in \cite{Cances:2012}. In this case, as the localized modes are exponentially decreasing, this convergence is exponentially fast with respect to the size of the truncated domain. In practice, this approach replaces the eigenvalue problem set in an unbounded domain to an approximated one set in a bounded domain. See \cite{Schmidt:2010} for numerical results. The main drawback of this strategy relies on the increase of the computational cost, especially when a mode is not well confined. We can mention also the fictitious source superposition method \cite{Botten:2006} and the reflective scattering matrix method.
\\\\
By adapting to eigenvalue problems the construction of Dirichlet-to-Neumann operators originally developed for scattering problems \cite{Fliss:2006,Fliss:2008,Fliss:2009}, we want to offer a rigorously justified alternative to existing methods. Compared to the supercell method, the DtN method allows us to reduce the numerical computation to a small neighborhood of the defect independently from the confinement of the computed guided modes. Moreover, as the method is exact, we improve the accuracy for non well-confined guided modes. Obviously, there is a price to be paid : the reduction of the problem leads to a non linear eigenvalue problem, of a fixed point nature. However, this difficulty has been already overcome for homogeneous open waveguides for which the DtN approach is well known \cite{Joly:1999,Pedreira:2001a,Pedreira:2001b}. 
\section{Model problem}
\noindent In order to describe the medium of study, let us first consider a two-dimensional periodic medium - \emph{the photonic crystal} - characterized by a coefficient $\rho_p$ (typically the square of the refraction index of the medium) which is
\begin{itemize}
	\item a $L^\infty function$ : $\exists \rho_-,\rho_+, \quad 0<\rho_-\leq \rho_p(x,y)\leq \rho_+$;
	\item periodic in the two directions
	\[
		\exists L_x,L_y>0,\;\forall n,m\in\Z,\;\forall(x,y)\in\R^2\quad \rho_p(x+nL_x,y+mL_y) = \rho_p(x,y)
	\]
\end{itemize} 
where the periods $L_x$ and $L_y$ are not necessarily equal.We introduce a one dimensional infinite perturbation - called the \emph{line defect} - in the $y-$direction in 
\[
	\Omega_0 \:=\: ]-a,a[ \:\times\: \mathbb{R}
\]
and characterized by a coefficient $\rho_0$ which is
\begin{itemize}
	\item a $L^\infty function$ satisfying $ 0<\rho_-\leq \rho_0\leq \rho_+$;
	\item periodic in the $y-$direction
	\[
		\exists L_y^0>0,\;\forall m\in\Z,\;\forall(x,y)\in\Omega^0\quad \rho_0(x,y+mL_y^0) = \rho_0(x,y)
	\]
	where $L_y$ and $L_y^0$ are commensurate. Without lack of generality, we suppose $L_y=L_y^0$.
\end{itemize}
The propagation medium is then characterized by the function $\rho$ defined by 
\begin{equation}\label{eq:def_rho}
	\begin{array}{r|l}
		\rho(x,y)=&\rho_p(x,y)\quad\text{in}\quad\R^2\setminus\Omega_0\\[3pt]
				&\rho_0(x,y)\quad\text{in}\quad\Omega_0.
			\end{array}
\end{equation}

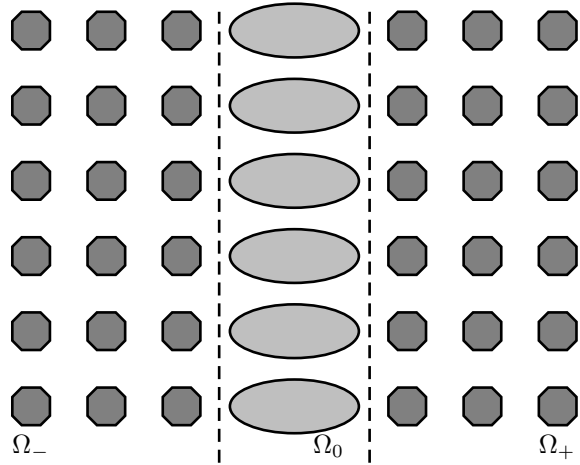
\begin{figure}[htbp]
\begin{center}
	\psset{unit=0.25cm} \psset{linewidth=1pt}
	\begin{pspicture}(-5,-6)(50,50)
	\pspolygon[fillstyle = solid,fillcolor=gray](2.5,-2)(3.5,-2)(4,-1.5)(4,-0.5)(3.5,0)(2.5,0)(2,-0.5)(2,-1.5)
	\pspolygon[fillstyle = solid,fillcolor=gray](6.5,-2)(7.5,-2)(8,-1.5)(8,-0.5)(7.5,0)(6.5,0)(6,-0.5)(6,-1.5)
	\pspolygon[fillstyle = solid,fillcolor=gray](10.5,-2)(11.5,-2)(12,-1.5)(12,-0.5)(11.5,0)(10.5,0)(10,-0.5)(10,-1.5)
	\psellipse[fillstyle = solid,fillcolor=lightgray](17,-1)(3.5,1.5)
	\pspolygon[fillstyle = solid,fillcolor=gray](22.5,-2)(23.5,-2)(24,-1.5)(24,-0.5)(23.5,0)(22.5,0)(22,-0.5)(22,-1.5)
	\pspolygon[fillstyle = solid,fillcolor=gray](26.5,-2)(27.5,-2)(28,-1.5)(28,-0.5)(27.5,0)(26.5,0)(26,-0.5)(26,-1.5)
	\pspolygon[fillstyle = solid,fillcolor=gray](30.5,-2)(31.5,-2)(32,-1.5)(32,-0.5)(31.5,0)(30.5,0)(30,-0.5)(30,-1.5)
	\pspolygon[fillstyle = solid,fillcolor=gray](2.5,2)(3.5,2)(4,2.5)(4,3.5)(3.5,4)(2.5,4)(2,3.5)(2,2.5)
	\pspolygon[fillstyle = solid,fillcolor=gray](6.5,2)(7.5,2)(8,2.5)(8,3.5)(7.5,4)(6.5,4)(6,3.5)(6,2.5)
	\pspolygon[fillstyle = solid,fillcolor=gray](10.5,2)(11.5,2)(12,2.5)(12,3.5)(11.5,4)(10.5,4)(10,3.5)(10,2.5)
	\psellipse[fillstyle = solid,fillcolor=lightgray](17,3)(3.5,1.5)
	\pspolygon[fillstyle = solid,fillcolor=gray](22.5,2)(23.5,2)(24,2.5)(24,3.5)(23.5,4)(22.5,4)(22,3.5)(22,2.5)
	\pspolygon[fillstyle = solid,fillcolor=gray](26.5,2)(27.5,2)(28,2.5)(28,3.5)(27.5,4)(26.5,4)(26,3.5)(26,2.5)
	\pspolygon[fillstyle = solid,fillcolor=gray](30.5,2)(31.5,2)(32,2.5)(32,3.5)(31.5,4)(30.5,4)(30,3.5)(30,2.5)
	\pspolygon[fillstyle = solid,fillcolor=gray](2.5,6)(3.5,6)(4,6.5)(4,7.5)(3.5,8)(2.5,8)(2,7.5)(2,6.5)
	\pspolygon[fillstyle = solid,fillcolor=gray](6.5,6)(7.5,6)(8,6.5)(8,7.5)(7.5,8)(6.5,8)(6,7.5)(6,6.5)
	\pspolygon[fillstyle = solid,fillcolor=gray](10.5,6)(11.5,6)(12,6.5)(12,7.5)(11.5,8)(10.5,8)(10,7.5)(10,6.5)
	\psellipse[fillstyle = solid,fillcolor=lightgray](17,7)(3.5,1.5)
	\pspolygon[fillstyle = solid,fillcolor=gray](22.5,6)(23.5,6)(24,6.5)(24,7.5)(23.5,8)(22.5,8)(22,7.5)(22,6.5)
	\pspolygon[fillstyle = solid,fillcolor=gray](26.5,6)(27.5,6)(28,6.5)(28,7.5)(27.5,8)(26.5,8)(26,7.5)(26,6.5)
	\pspolygon[fillstyle = solid,fillcolor=gray](30.5,6)(31.5,6)(32,6.5)(32,7.5)(31.5,8)(30.5,8)(30,7.5)(30,6.5)
	\pspolygon[fillstyle = solid,fillcolor=gray](2.5,10)(3.5,10)(4,10.5)(4,11.5)(3.5,12)(2.5,12)(2,11.5)(2,10.5)
	\pspolygon[fillstyle = solid,fillcolor=gray](6.5,10)(7.5,10)(8,10.5)(8,11.5)(7.5,12)(6.5,12)(6,11.5)(6,10.5)
	\pspolygon[fillstyle = solid,fillcolor=gray](10.5,10)(11.5,10)(12,10.5)(12,11.5)(11.5,12)(10.5,12)(10,11.5)(10,10.5)
	\pspolygon[fillstyle = solid,fillcolor=gray](22.5,10)(23.5,10)(24,10.5)(24,11.5)(23.5,12)(22.5,12)(22,11.5)(22,10.5)
	\psellipse[fillstyle = solid,fillcolor=lightgray](17,11)(3.5,1.5)
	\pspolygon[fillstyle = solid,fillcolor=gray](26.5,10)(27.5,10)(28,10.5)(28,11.5)(27.5,12)(26.5,12)(26,11.5)(26,10.5)
	\pspolygon[fillstyle = solid,fillcolor=gray](30.5,10)(31.5,10)(32,10.5)(32,11.5)(31.5,12)(30.5,12)(30,11.5)(30,10.5)
	\pspolygon[fillstyle = solid,fillcolor=gray](2.5,14)(3.5,14)(4,14.5)(4,15.5)(3.5,16)(2.5,16)(2,15.5)(2,14.5)
	\pspolygon[fillstyle = solid,fillcolor=gray](6.5,14)(7.5,14)(8,14.5)(8,15.5)(7.5,16)(6.5,16)(6,15.5)(6,14.5)
	\pspolygon[fillstyle = solid,fillcolor=gray](10.5,14)(11.5,14)(12,14.5)(12,15.5)(11.5,16)(10.5,16)(10,15.5)(10,14.5)
	\psellipse[fillstyle = solid,fillcolor=lightgray](17,15)(3.5,1.5)
	\pspolygon[fillstyle = solid,fillcolor=gray](22.5,14)(23.5,14)(24,14.5)(24,15.5)(23.5,16)(22.5,16)(22,15.5)(22,14.5)
	\pspolygon[fillstyle = solid,fillcolor=gray](26.5,14)(27.5,14)(28,14.5)(28,15.5)(27.5,16)(26.5,16)(26,15.5)(26,14.5)
	\pspolygon[fillstyle = solid,fillcolor=gray](30.5,14)(31.5,14)(32,14.5)(32,15.5)(31.5,16)(30.5,16)(30,15.5)(30,14.5)
	\pspolygon[fillstyle = solid,fillcolor=gray](2.5,18)(3.5,18)(4,18.5)(4,19.5)(3.5,20)(2.5,20)(2,19.5)(2,18.5)
	\pspolygon[fillstyle = solid,fillcolor=gray](6.5,18)(7.5,18)(8,18.5)(8,19.5)(7.5,20)(6.5,20)(6,19.5)(6,18.5)
	\pspolygon[fillstyle = solid,fillcolor=gray](10.5,18)(11.5,18)(12,18.5)(12,19.5)(11.5,20)(10.5,20)(10,19.5)(10,18.5)
	\psellipse[fillstyle = solid,fillcolor=lightgray](17,19)(3.5,1.5)
	\pspolygon[fillstyle = solid,fillcolor=gray](22.5,18)(23.5,18)(24,18.5)(24,19.5)(23.5,20)(22.5,20)(22,19.5)(22,18.5)
	\pspolygon[fillstyle = solid,fillcolor=gray](26.5,18)(27.5,18)(28,18.5)(28,19.5)(27.5,20)(26.5,20)(26,19.5)(26,18.5)
	\pspolygon[fillstyle = solid,fillcolor=gray](30.5,18)(31.5,18)(32,18.5)(32,19.5)(31.5,20)(30.5,20)(30,19.5)(30,18.5)
	\psline[linestyle=dashed](13,-4)(13,20)\psline[linestyle=dashed](21,-4)(21,20)
	\put(2,-3.5){$\Omega_-$}\put(18,-3.5){$\Omega_0$}\put(30,-3.5){$\Omega_+$}
	\end{pspicture}
\caption{Domain of propagation : typically $\rho=1$ in the white region, $\rho=2$ in the dark grey regions and $\rho=3$ in the light grey regions.}\label{fig:omega}
\end{center}
\end{figure}
\begin{remark}[Some extensions]\label{rem:extension_rho}~\\
	1. We can consider a more general medium of propagation than $\Omega=\R^2$. It could consist for example of $\R^2$ minus a periodic set of holes. The only assumption is that the periodicity property of $\Omega$ has to the be the same as $\rho$.\\
	2. Even if the spectral properties are different, the analysis and the method can be extended to coefficient defined by
	 \[
		\begin{array}{r|ll}
			\rho(x,y)=&\rho_p^-(x,y)&\text{in}\quad]-\infty,-a[\times\R\\[3pt]
					&\rho_0(x,y)&\text{in}\quad\Omega_0\\[3pt]
					&\rho_p^+(x,y)&\text{in}\quad]a,+\infty[\times\R.
				\end{array}
	\]
where $\rho_p^-$ and $\rho_p^+$ have the same properties than $\rho_p$ with the same period in the $y-$direction ($L_y$) and not necessarily the same in the $x$-direction ($L_x^-$ and $L_x^+$).
\end{remark}~\\
\begin{remark}[Some open questions]\label{rem:openquestions_rho}
The case where the line defect has not the same periodicity properties as the photonic crystal and the case where the line defect is not introduced along one of the direction of periodicity of the photonic crystal are out of the scope of the present paper.
\end{remark}~\\\\
The propagation model is a simple 2D space-(${\bfx}=(x,y)$) time harmonic scalar wave equation (corresponding for example to transverse electric (TE) mode electromagnetic wave in two dimensions)
\begin{equation}\label{eq:wave}\tag{$\mathcal{P}$}
	-\triangle w-\rho\,\omega^2\, w =0,\quad\text{in}\;\Omega,
\end{equation}
where $\omega\in \R^+$ is the frequency. In the rest of the article, $\omega^2$ will be considered as the spectral parameter.\\
\begin{remark}\label{rem:generalisation}
	The results developed in this article can be easily extended to more general elliptic operator $u\mapsto \nabla\cdot(\mu\nabla u)$ where $\mu$ is line perturbation of a periodic function.
\end{remark}~\\\\
Using the Floquet-Bloch theory, we could show that the spectrum of the operator $-\triangle/\rho$ could be deduced from the study of the guided modes. A guided mode of this problem is by definition a solution $w\neq 0$ to \eqref{eq:wave}  which can be written in the form
\begin{equation}\label{eq:mode}
	w(x,y)=v(x,y)\,e^{\imath\beta y}
\end{equation}
where, in full generality
\begin{itemize}
	\item $\beta$ - called the quasi period - is in $]-\pi/L_y,\pi/L_y[$;
	\item $v$ is periodic in the $y-$direction with period $L_y$ and
	\[
		v\big|_{B}\in H^1(B)\quad\text{where}\;B=\R\times ]-\frac{L_y}{2},\frac{L_y}{2}[
	\]
\end{itemize}
We will identify in the following any periodic function in $\Omega$ to its restriction to the period $B$.~\\
\begin{remark}\label{rem:beta}
In full generality, the coefficient $\beta$ is in $\R$ but it is enough to consider it only in $]-\pi/L_y,\pi/L_y[$. Indeed, it is easy to see that if there exists a mode for $\beta\in\R$, 
	\[
		w(x,y,t)=v(x,y)\,e^{\imath\beta y},\;\text{with $v$ $L_y$-periodic and }\omega\in \R^+,
	\]
it corresponds to a mode for $\beta\pm\pi/L_y$ of the form
	\[
		w(x,y)=\tilde{v}(x,y)\,e^{\imath(\beta\pm\pi/L_y) y}\;\text{with} \; \tilde{v}(x,y) = v(x,y) e^{\mp\imath\pi/L_y y}\text{ which is $L_y$-periodic}.
	\]
\end{remark}
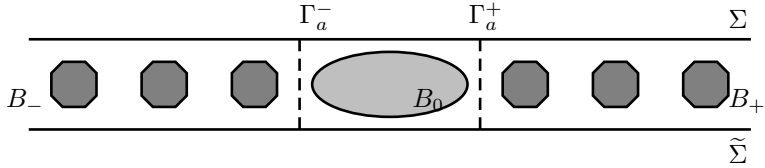
\begin{figure}[htbp]
\begin{center}
	\psset{unit=0.3cm} \psset{linewidth=1pt}
	\begin{pspicture}(-5,-5)(50,10)
	\pspolygon[fillstyle = solid,fillcolor=gray](2.5,2)(3.5,2)(4,2.5)(4,3.5)(3.5,4)(2.5,4)(2,3.5)(2,2.5)
	\pspolygon[fillstyle = solid,fillcolor=gray](6.5,2)(7.5,2)(8,2.5)(8,3.5)(7.5,4)(6.5,4)(6,3.5)(6,2.5)
	\pspolygon[fillstyle = solid,fillcolor=gray](10.5,2)(11.5,2)(12,2.5)(12,3.5)(11.5,4)(10.5,4)(10,3.5)(10,2.5)
	\psellipse[fillstyle = solid,fillcolor=lightgray](17,3)(3.5,1.5)
	\pspolygon[fillstyle = solid,fillcolor=gray](22.5,2)(23.5,2)(24,2.5)(24,3.5)(23.5,4)(22.5,4)(22,3.5)(22,2.5)
	\pspolygon[fillstyle = solid,fillcolor=gray](26.5,2)(27.5,2)(28,2.5)(28,3.5)(27.5,4)(26.5,4)(26,3.5)(26,2.5)
	\pspolygon[fillstyle = solid,fillcolor=gray](30.5,2)(31.5,2)(32,2.5)(32,3.5)(31.5,4)(30.5,4)(30,3.5)(30,2.5)
	\psline(1,1)(33,1)\psline(1,5)(33,5)\psline[linestyle=dashed](21,1)(21,5)
	\psline[linestyle=dashed](13,1)(13,5)
	\put(0,2){$B_-$}\put(32,-0.5){$\widetilde{\Sigma}$}\put(20.5,5.7){$\Gamma^+_a$}\put(13,5.7){$\Gamma^-_a$}\put(32,5.5){$\Sigma$}\put(18,2){$B_0$}\put(32,2){$B_+$}
	\end{pspicture}
\caption{The band $ B$ - one period in the $y$-direction of the domain $\Omega$ represented in Figure~\ref{fig:omega}.}\label{fig:bandB}
\end{center}
\end{figure}
Replacing the expression \eqref{eq:mode} in \eqref{eq:wave} we see easily that finding the guided modes corresponds to finding couples $(\omega^2,\beta)$ such that there exists
$v\in H^1(B)$, $v\neq 0$ solution of
\begin{equation}\label{eq:pb_modes1}
\begin{array}{|l}
	\dsp -\frac{1}{\rho}\triangle v - \frac{2\imath\beta}{\rho} \nabla v + \frac{\beta^2}{\rho} v= \omega^2v,\quad\text{in}\;B\\[3pt]
	\dsp v\big|_{\Sigma} = v\big|_{\widetilde{\Sigma}},\quad
	\dsp\partial_y\, v\big|_{\Sigma} = \partial_y\, v\Big|_{\widetilde{\Sigma}}
\end{array}
\end{equation}
where (see Figure~\ref{fig:bandB}) $
	\Sigma =\R\times\big\{\frac{L_y}{2}\big\}$ and 	$
		\widetilde{\Sigma} =\R\times\big\{- \frac{L_y}{2}\big\}.$ 
\\\\
We can rewrite the previous problem using $u(x,y) = v(x,y)e^{\imath\beta y}$ and then finding the guided modes is equivalent to finding couples $(\omega^2,\beta)$ such that there exists $u\in H^1(B)$, $u\neq 0$ solution of
\begin{equation}\label{eq:pb_modes2}
\begin{array}{|l}
	\dsp -\frac{1}{\rho}\triangle u = \omega^2u,\quad\text{in}\;B\\[3pt]
	\dsp u\lrres_{\Sigma} = e^{\imath \beta}\,u\lrres_{\widetilde{\Sigma}},\quad
	\dsp\partial_y\, u\lrres_{\Sigma} = e^{\imath \beta}\,\partial_y\, u\lrres_{\widetilde{\Sigma}}
\end{array}
\end{equation}
To characterize the guided modes, there exist two different formulations : the $\omega-$formulation in which $\omega$ is fixed and $\beta$ is looked for and the $\beta-$formulation in which $\beta$ is fixed and $\omega$ is looked for. The first one leads to a quadratic eigenvalue problem in the band $B$ (it is obvious with the formulation \eqref{eq:pb_modes1} of the problem) and the second one to an eigenvalue problem in the band $B$.
\\\\
The classical approach for solving this problem set in an unbounded domain is the Supercell Method which is based on the exponentially decreasing property of the mode in the $x-$ direction (see Theorem \ref{th:decay_eigenvec} where we remind this result) and consists in truncating the band $B$ far enough. One (artificial) parameter of the method is then the position of the truncature position. The main advantage is that this approach leads to an eigenvalue problem (for the $\beta$-formulation) or a quadratic one (for the $\omega-$formulation) set on a bounded domain (the truncated one). The main drawback of this strategy relies on the increase of the computational cost when a mode is not well confined. 
\\\\
Here we propose a novel method  based on a DtN approach, which is originally developed for scattering problems. This method offers a rigorously justified alternative to existing methods.This approach allows us to reduce the numerical computation to a small neighborhood of the defect independently from the confinement of the computed guided mode. The main advantage is that the method is exact so we improve the accuracy for non well-confined guided modes and could deduce the behavior of the dispersive curves - $\beta\mapsto\omega(\beta)$ where $(\beta,\omega(\beta))$ is solution of \eqref{eq:pb_modes1} - near the edges of the gaps. The main drawback is that the reduction of the problem leads to a nonlinear eigenvalue problem, of a fixed point nature. To simplify the presentation, we choose here the $\omega-$formulation but the method extends to the other formulation - which could be more adapted for dispersive media for example. $\rho=\rho(\omega)$).
\\\\
The article is organized as follows. We will remind in Section 3 the spectral properties of the locally perturbed periodic operator involved in the study. Section 4 deals with the DtN approach and present the non linear eigenvalue problem which has to be solved to compute the guided modes. It is the most important and original part of the article. Section 5 is devoted to numerical results and Section 6 to give some conclusions and perspectives. 

\section{Spectral theory results for the problem} 
\label{sec:spectral_theory_results_for_the_problem}

\noindent We focus on the guided modes defined in \eqref{eq:pb_modes2} and then reduce the problem to the following one : \\\\
$\dsp\text{For any }\beta\in ]-\frac{\pi}{L_y},\frac{\pi}{L_y}[,$
 \begin{equation}\label{eq:eigpb}\tag{$\mathcal{E}^\beta$}
\dsp \text{find}\;\omega^2\in \R^+,\quad
		\text{s.t.}\;\exists u \in H^1(B),\;u\neq 0,\quad
 \dsp A(\beta)u=\omega^2 u
 \end{equation}
where 
\[
\begin{array}{|l}
	\dsp A(\beta) = -\frac{1}{\rho}\triangle\\
	\dsp D(A(\beta)) = \big\{u\in H^1(\triangle,B),\quad
	u|_{\Sigma}=e^{\imath\beta}u|_{\widetilde{\Sigma}}\quad\text{and}\quad
	\partial_y u|_{\Sigma}=e^{\imath\beta}\partial_y u|_{\widetilde{\Sigma}} \big\}.
\end{array}
\]
with $H^1(\triangle, B) = \{u \in H^1(B),\; \Delta u \in L^2(B)\}$.\\\\The problem is then reduced to the determination of eigenvalues of the locally perturbed periodic operator $A(\beta)$, for any $\beta\in ]-{\pi}/{L_y},{\pi}/{L_y}[.$\\\\
We introduce for the following the operator with perfectly periodic coefficients
\begin{equation}\label{eq:A_p_ky}
	A_p(\beta) = -\frac{1}{\rho_p}\,\triangle,\quad D\lrpar{A_p(\beta)} = D(A(\beta))
\end{equation}
Using \cite{Kuchment:1993,Figotin:1997} and the Weyl's theorem \cite[Vol. IV, Theorem XIII-14]{Reed:1972}, we describe in the following proposition the essential spectrum of the operator $A(\beta)$.\\\\
\begin{proposition}[Essential spectrum of $A(\beta)$]\label{prop:essential_spectrum}
The operator $A(\beta)$ is selfadjoint in $H=L^2(B,\rho dxdy)$, positive and its essential spectrum, denoted $\sigma_{ess}(\beta)$, satisfies
\[
	\sigma_{ess}(\beta)=\sigma\left(A_p(\beta)\right)=\R\setminus\bigcup_{1\leq n\leq N(\beta)} ]a_n(\beta),b_n(\beta)[
\]	
where $A_p(\beta)$ is the operator with periodic coefficient defined by \eqref{eq:A_p_ky}, $N(\beta)$ ($0\leq N(\beta)\leq +\infty$) is the number of gaps and for any $n$, $0\leq a_n(\beta)<b_n(\beta)$. 
\end{proposition}~\\
\begin{proof}
The operator $A_p(\beta)$ is selfadjoint in $L^2(B,\rho_p dxdy)$ and positive, the operator $A(\beta)$ is closed in $L^2(B,\rho_p dxdy)$ (and obviously selfadjoint and positive in $H$). Moreover, $(A(\beta)+1)^{-1}-(A_p(\beta)+1)^{-1}$ is compact (by the compactness of the embedding of $H^1(B_0) \hookrightarrow L^2(B_0)$ where $B_0=B\cap \Omega_0$). Using \cite[Vol. IV, Theorem XIII-14]{Reed:1972}, the essential spectrum of $A(\beta)$ coincide with the essential spectrum of $A_p(\beta)$
	\begin{equation}\label{eq:weyl}
		\sigma_{ess}\left(\beta\right) = \sigma_\text{ess}\left(A_p(\beta)\right).
	\end{equation}
	Moreover, one of the main result of the Floquet-Bloch theory is (see \cite{Kuchment:1993} for more details) that the spectrum of $A_p(\beta)$, $\sigma\left(A_p(\beta)\right)$ is reduced to its essential spectrum, $\sigma_\text{ess}\left(A_p(\beta)\right)$ and is given by
	\begin{equation}\label{eq:spectre_Ap}
		\sigma\left(A_p(\beta)\right) = \sigma_\text{ess}\left(A_p(\beta)\right) = \bigcup_{k\in]-\pi/L_x,\pi/L_x]} \sigma\left(A_p(\beta,k)\right)
	\end{equation}
	where 
	\[\begin{array}{|l}
		\dsp A_p(\beta,k) = -\frac{1}{\rho_p}\,\triangle,\\[5pt] \dsp D\lrpar{A_p(\beta,k)} = \lracco{u\in H^1(\triangle, \calc),\;
			\begin{array}{|l}
				\dsp u\lrres_{x=L_x/2} = e^{\imath kL_x}\,u\lrres_{x=-L_x/2},\\
				\dsp\partial_x u\lrres_{x=L_x/2} = e^{\imath kL_x}\,\partial_x u\lrres_{x=-L_x/2}\\
				\dsp u\lrres_{y=L_y/2} = e^{\imath \beta L_y}\,u\lrres_{y=-L_y/2},\\
				\dsp\partial_y u\lrres_{y=L_y/2} = e^{\imath \beta L_y}\,\partial_yu\lrres_{y=-L_y/2}
			\end{array}
		}.
	\end{array}
	\]
	with $\calc = ]-L_x/2,L_x/2[\times]-L_y/2,L_y/2[$ and $H^1(\triangle, \calc) = \{u \in H^1(\calc),\; \Delta u \in L^2(\calc\}$. \\\\
	Moreover, for any $k$ in $]-\pi/L_x,\pi/L_x]$, $A_p(\beta,k)$ is a self-adjoint positive operator with compact resolvent so its spectrum is purely discrete
	\[
		0<\omega_1(\beta,k)\leq \omega_2(\beta,k)\leq \ldots \leq \omega_n(\beta,k)<\ldots\quad\text{with}\;\lim_{n\rightarrow +\infty}\omega_n(\beta,k) = +\infty
	\]
	and we can find $\lrpar{e_n(\beta,k)}$ a corresponding family of eigenvectors which is an Hilbert basis of $L^2(\calc)$.
	\\\\
	Using the min-max principle to $\tilde{A}_p(\beta,k)$, we could show that the so called dispersive curve
	\begin{equation}\label{eq:omega_n_cont}
		(\beta,k)\mapsto\omega_n(\beta,k)\;\text{is continuous.}
	\end{equation}
	Finally, we deduce from \eqref{eq:weyl}, \eqref{eq:spectre_Ap} and \eqref{eq:omega_n_cont} that
	\[
	\sigma_{ess}(\beta)=\sigma\left(A_p(\beta)\right)=\bigcup_{n\in\N}\omega_n([-\pi,\pi[,\beta)
	\]
\end{proof}
~\\\\
\begin{remark}
In Section \ref{sec:ess_spectrum}, we will give another characterization of the essential spectrum with a by-product of the method.
\end{remark}~\\\\
As described in the Introduction, only sufficient conditions of existence are known and finiteness of the number of gaps is conjectured in the general case and proven for particular cases. 
Let us suppose in the following that at least one gap exists ($N(\beta)\geq 1$). Using the theory of selfadjoint operators (\cite{Reed:1972}), we deduce that\\
\begin{proposition}[Discrete spectrum of $A(\beta)$]
	The spectrum of $A(\beta)$ inside the gaps consists only of isolated eigenvalues of finite multiplicity, which can accumulate only at the edges of the gap. 
\end{proposition}~\\\\
One should ask if there is a way to ensure the rise of at least one eigenvalue in gaps of $A(\beta)$. Figotin and Klein have given in \cite{Figotin:1997} sufficient condition on the defect (supposed homogeneous) and the gaps to introduce eigenvalues in any gaps of $A(\beta)$. Using asymptotic arguments, Nazarov and co-workers \cite{Nazarov:2011} have studied sufficient conditions for existence of eigenvalue for elastic waveguides.
\\\\
We are interested now, in characterizing and then computing the eigenvalues $(\lambda_m(\beta))_m$,
\[
	0\leq \lambda_1(\beta)\leq \lambda_2(\beta)\leq\ldots\leq\lambda_{M(\beta)}(\beta),\quad 0\leq M(\beta)\leq +\infty
\]
if they exist (we suppose then $M(\beta)\geq 1$), which are in the gaps of the essential spectrum (see \cite{Bonnet:2001} -optical waveguides- and \cite{Kuchment:2003} -photonic crystal waveguides- for existence of eigenvalues inside gaps) :
~\\
\begin{proposition}[Properties of each $\lambda_m(\beta)$]
	The dispersion curves $\beta\mapsto\lambda_m(\beta)$ are $2\pi/L_y$-periodic, even and continuous.
\end{proposition}~\\
\begin{proof}
	\begin{itemize}
		\item $\beta\mapsto\lambda_m(\beta)$ is $2\pi/L_y$-periodic : by definition 
		\[
			A(\beta+2\pi/L_y) = A(\beta)\quad\text{and}\quad D\left(A(\beta+2\pi/L_y)\right) = D\left(A(\beta)\right).
		\]
		\item $\beta\mapsto\lambda_m(\beta)$ is even because
		\[
			A(\beta) = A(-\beta)\quad\text{and}\quad \overline{D\left(A(\beta)\right)} = D\left(A(-\beta)\right).
		\]
		and $A(\beta)$ is a self-adjoint positive operator.
		\item The continuity with respect to $\beta$ of the eigenvalues sorted in ascended order is due to the analyticity of the operator $A(\beta)$ with respect to $\beta$ \cite{kato}. 
	\end{itemize}
\end{proof}~\\\\
~\noindent We deduce in particular that it is sufficient to study the dispersive curves for $\beta\in [0,\pi/L_y]$.\\\\
We study now the properties of the eigenvectors associated to the eigenvalues. By definition, they are in $L^2(B)$ but we can be more precise : they decay exponentially fast far from the defect, with a rate depending on the distance from the eigenvalue to the edges of the gap. It is a result based on \cite{Figotin:1996c} and shown in \cite{Figotin:1997}.\\
\begin{theorem}[Exponential decay of the eigenvectors]\label{th:decay_eigenvec}
	Let $\beta\in ]-\pi/L_y,\pi/L_y]$. For any eigenvalue $\lambda_n(\beta)$ of $A(\beta)$, the associated eigenvectors $\varphi_n(\beta;\cdot)$ satisfies
	\begin{multline}
		\forall n\in\llbracket1,M(\beta)\rrbracket,\quad\exists C_1,C_2>0,\;\exists a>0,\forall (x,y)\in B,\;\abs{x}<a,\quad \\\dsp\abs{\varphi_n(\beta;x,y)}\leq \frac{C_1}{\text{dist}\left(\lambda_n(\beta),\sigma_{ess}(A(\beta))\right)}e^{\left(-C_2\,\text{dist}\left(\lambda_n(\beta),\sigma_{ess}(A(\beta))\right)\abs{x}\right)}
	\end{multline}
\end{theorem}
This property is exactly the one which encourages to use the Supercell method. Indeed, this method consists in approaching the eigenvalues and the associated eigenvectors by the ones of a truncated problem in the $x-$direction with periodic conditions. If the eigenvectors are exponentially decreasing in the $x-$direction, they will satisfy almost periodic conditions far enough from the perturbation and then be {\it almost-}eigenvectors for the {\it almost-}same eigenvalues of an operator defined from the truncated domain. The main advantage of this approach it that it leads to eigenvalue problem for a $\beta$-formulation and a quadratic one for a $\omega-$formulation, both of them set on a bounded domain. However, we could note several drawbacks.
\begin{enumerate}
	\item The essential spectrum of $A(\beta)$ has to be computed initially.
	\item For a fixed $\beta$, it seems important to have an estimation of the distance between the not yet computed eigenvalue and the essential spectrum of $A(\beta)$ to choose a relevant truncated domain. 
	\item The size of the truncated domain depends on $\beta$ and on the distance between the eigenvalue and the essential spectrum. If the eigenvalue approaches more and more the essential spectrum of $A(\beta)$ when $\beta$ varies, the corresponding eigenvector becomes less and less confined (less and less exponentially decreasing) and then the truncated domain has to be bigger and bigger. This can increase dramatically the computational cost.
	\item Because, this method is based on the exponential decay of the eigenvectors, it cannot describe the behavior of the dispersive curves (the eigenvalues as functions of $\beta$, $\beta\mapsto\lambda_n(\beta)$) near the edges of the essential spectrum. 
\end{enumerate}
We propose now a novel method based on a DtN approach which overcome these disadvantages because it is an exact method (in the sense which is precised in the next section). However, its main drawback is that it leads to a nonlinear eigenvalue problem of a fixed point nature.
\section{The non linear eigenvalue problem}
\subsection{The DtN approach} 
\label{sub:the_dtn_approach}

\noindent Let $\beta\in ]-\pi/L_y,\pi/L_y]$ and let us suppose from this point that $\alpha^2\notin \sigma_{ess}(\beta)$. We want here to write an equivalent problem to the problem \eqref{eq:eigpb} which is set on a bounded domain.\\\\
For the sequel, it is essential to introduce functional spaces appearing naturally in the study. Let us remind that $B^\pm = B\cap\Omega^\pm$, $\Sigma^\pm=\Sigma\cap\Omega^\pm$, $\widetilde{\Sigma}^\pm=\widetilde{\Sigma}\cap\Omega^\pm$ and $\Gamma^\pm_a=\{\pm a\}\times ]-L_y/2,L_y/2[$ (see Figure~ \ref{fig:bandB}).\\\\We start from smooth quasi-periodic functions in $B^\pm$:
\[
		C^\infty_{\beta}(B^\pm) = \Big\{u = \tilde{u}\big|_{B^\pm} ,\; \tilde{u} \in C^\infty(\Omega^\pm), \quad
\tilde{u}(x, y+L) = e^{\imath \beta L_y}\; \tilde{u}(x,y)
\Big\}.
\]
Let $H^1_{\beta}(B^\pm)$ be the closure of $C^\infty_{\beta}(B^\pm)$ 
in $H^1(B^\pm)$
\[
H^1_{\beta}(B^\pm) = \Big\{u \in H^1(B^\pm),\quad u\big|_{\Sigma^\pm} = e^{\imath \beta L_y}\;u\big|_{\widetilde{\Sigma}^\pm}\Big\}
\]
where in the last equation we have identified the spaces $H^{1/2}(\Sigma^\pm)$ and $H^{1/2}(\widetilde{\Sigma}^\pm)$. As $H^1_\beta(B^\pm)$ is a closed subspace of $H^1(B^\pm)$, we equip it with the norm of $H^1(B^\pm)$. Let $H^1_{\beta}(\triangle,B^\pm)$ be the closure of $
H^1(\triangle, B^\pm) = \{u \in H^1(B^\pm),\; \Delta u \in L^2(B^\pm)\}.
$
\[
H^1_{\beta}(\triangle,B^\pm) = \Big\{u \in H^1(\triangle,B^\pm) \cap H^1_{\beta}(B^\pm),	\quad
\dsp \frac{\partial u}{\partial y}\Big|_{\Sigma^\pm} = e^{\imath \beta L_y}\; \frac{\partial u}{\partial y}\Big|_{\widetilde{\Sigma}^\pm}
\Big\}.
\]
where in the last equation we have identified the spaces $H^{1/2}_{00}(\Sigma^\pm)'$ and $H^{1/2}_{00}(\widetilde{\Sigma}^\pm)'$.\\[12pt]
The space $H^{1/2}_{\beta}(\Gamma_a^\pm)$ is defined by 
\[
	H^{1/2}_{\beta}(\Gamma_a^\pm) = \gamma_0^\pm\Big(H^1_{\beta}(B^\pm)\Big)
\]
where $\gamma_0^\pm\in\mathcal{L}(H^1(B^\pm),H^{1/2}(\Gamma_a^\pm))$ is the trace map on $\Gamma_a^\pm$ : $
	\forall\, u\in H^1(B^\pm),\; \gamma_0^\pm u = u|_{\Gamma_a^\pm}.$
$H^{1/2}_{\beta}(\Sigma_0)$ is then a dense subspace of $H^{1/2}(\Sigma_0)$ and 
the injection from $H^{1/2}_{\beta}(\Gamma_a^\pm)$ onto $H^{1/2}(\Gamma_a^\pm)$ is continuous.\\\\
We define $H^{-1/2}_{\beta}(\Gamma_a^\pm)$ as the dual of $H^{1/2}_{\beta}(\Gamma_a^\pm)$.
\\\\
Finally, the trace application $\gamma_1^\pm\in\mathcal{L}(H^1(\triangle,B^\pm),H^{1/2}_{(a,a)}(\Gamma_a^\pm)^\prime)$ defined by :
\[
	\forall\, u\in H^1(\triangle,B^\pm),\quad \gamma_1^\pm u = \frac{\partial u}{\partial x}\Big|_{\Gamma_a^\pm}
\]
is a continuous application from $H^1_{\beta}(\triangle,B^\pm) \text{ onto } 	H^{-1/2}_{\beta}(\Gamma_a^\pm)$ and we can show that
\[
	H^{-1/2}_{\beta}(\Gamma_a^\pm) = \gamma_1\Big(H^1_{\beta}(\triangle,B^\pm)\Big).
\]
Let us now introduce the two half-band problems: for any $\beta$ and $\alpha$ and for any given $\varphi$ in $H^{1/2}_\beta(\Gamma^\pm_a)$ 
\begin{equation}\label{eq:pbplus}\tag{$\mathcal{P}^\pm$}
		\text{Find}\;u^\pm\in\,H^1_\beta(\triangle,B^\pm),\quad\quad
	\begin{array}{|l}
		-\triangle u^\pm-\rho_p\alpha^2 u^\pm = 0\quad\text{in}\;B^\pm \\[3pt]
		u|_{\Gamma_a^\pm} = \varphi.
	\end{array}
\end{equation}
\begin{theorem}[Well-posedness of the problem \eqref{eq:pbplus}]\label{th:pb+-_bien}~\\~
	If $\alpha^2\notin \sigma_{ess}(\beta)$, the problem \eqref{eq:pbplus} is well-posed in $H^1_\beta$ except for a countable set of frequencies which depends on $\beta$.\\\\
	If the periodicity cell is symmetric with respect to the axis $x=0$ and if $\alpha^2\notin \sigma_{ess}(\beta)$, the problem \eqref{eq:pbplus} is always well-posed in $H^1_\beta$.
\end{theorem}~\\
\begin{proof}
It is enough to show the result for $(\mathcal{P}^+)$. We introduce the following operators
	\[
	\begin{array}{|l}
		\dsp A_{D,+}(\beta) = -\frac{1}{\rho_p}\triangle,\quad D(A_{D,+}(\beta)) = \left\{u^+\in H^1_\beta(B^+),\;u^+\big|_{\Gamma_a^+} = 0\right\}\\[5pt]
		\dsp A_{D,-}(\beta) = -\frac{1}{\rho_p}\triangle,\quad D(A_{D,-}(\beta)) = \left\{u^-\in H^1_\beta(\Omega\setminus B^+),\;u^-\big|_{\Gamma_a^+} = 0\right\}\\[5pt]
		\dsp A_{D}(\beta) = -\frac{1}{\rho_p}\triangle,\quad D(A_{D}(\beta)) = \left\{u\in H^1_\beta(\Omega),\;u\big|_{\Gamma_a^+} = 0\right\}
	\end{array}
	\]
It is easy to see that the spectrum of $A_D$ is the union of the spectrums of $A_{D,+}$ and $A_{D,-}$. Moreover the resolvant of $A_D$ is a compact perturbation of the resolvant of $A$ so $A_D$ and $A$ has the same essential spectrum. These two properties implies that $A_{D,+}$ has its essential spectrum included in the spectrum of $A$ but this operator may have a countable set of eigenvalues. In conclusion, $\alpha^2\notin\sigma_{ess}(\beta)$, the problem $(\mathcal{P}^+)$ is well posed if $\alpha^2$ is not an eigenvalue of $A_{D,+}$.
	\\\\
		If the periodicity cell is symmetric with respect to the axis $x=0$, the operator $A_{D,+}$ has no eigenvalues. Indeed, if an eigenvector exists, we could construct by antisymmetry with respect to $\Gamma_a^+$ an eigenvector of $A$, which is not possible. 
\end{proof}~\\\\
\begin{remark}[Robin-to-Robin operators instead of DtN operators]\label{rem:RtR_vs_DtN}
	The countable set of frequencies for which the half-band problems are not well-posed is introduced by our method of construction of DtN operators. If we use Robin-to-Robin operators instead of Dirichlet-to-Neumann operators, the problems are always well posed when the frequency $\alpha^2\notin \sigma_{ess}(\beta)$. We choose to use DtN operators in the present paper to simplify the presentation of the method (see \cite{Fliss:2009,Fliss:2009b} for more details on Robin-to-Robin operators)
\end{remark}
~\\\\
Suppose from this point that the problems $(\mathcal{P}^+)$ and $(\mathcal{P}^-)$ are well posed. We denote by $u^+(\beta,\alpha;\varphi)$ and $u^-(\beta,\alpha;\varphi)$ the respective unique solutions.\\\\ 
The DtN operators $\Lambda^\pm(\beta,\alpha)\in \mathcal{L}(H^{1/2}_\beta(\Gamma^\pm_a),H^{-1/2}_\beta(\Gamma^\pm_a))$ are then defined by
\[
	\forall\varphi\in H^{1/2}_\beta(\Gamma^\pm_a),\quad \Lambda^\pm (\beta,\alpha)\varphi = \mp\partial_x\,u^\pm(\beta,\alpha;\varphi)
\]
or else $\forall\varphi,\psi\in H^{1/2}_\beta(\Gamma^\pm_a)$
\begin{multline}\label{eq:def_dtn}
	<\Lambda^\pm (\beta,\alpha)\varphi,\psi>_{\Gamma^\pm_a} = \int_{B^\pm}\nabla u^\pm(\beta,\alpha;\varphi)\cdot \nabla u^\pm(\beta,\alpha;\psi)\\-\alpha^2 \int_{B^\pm}\rho_p u^\pm(\beta,\alpha;\varphi) u^\pm(\beta,\alpha;\psi)
\end{multline}
where $<\cdot,\cdot>_{\Gamma^\pm_a}$ denotes the duality product between $H^{-1/2}_{\beta}(\Gamma^\pm_a)$ and $H^{1/2}_{\beta}(\Gamma^\pm_a)$.
We have then by definition the continuity properties of the DtN operators
\begin{proposition}\label{prop:dtn_continuous}
	For any $\beta$, the DtN operators $\Lambda^\pm(\beta,\alpha)$ are continuous from $H^{1/2}_{\beta}(\Gamma^\pm_a)$ onto $H^{-1/2}_{\beta}(\Gamma^\pm_a)$ and are norm continuous with respect to $\alpha$.
\end{proposition}
\begin{proof}
We introduce the operator 
	\[
		A_{D,\pm}(\beta) = -\frac{1}{\rho_p}\triangle,\quad D(A_{D,\pm}(\beta)) = \left\{u^\pm\in H^1_\beta(B^\pm),\;u^\pm\big|_{\Gamma_a^\pm} = 0\right\}.
	\]
We have seen in Theorem \ref{th:pb+-_bien} that if $\alpha^2\notin\sigma_{ess}(\beta)$ and $\alpha^2\notin\sigma_{d,\pm}(\beta)$ where $\sigma_{d,\pm}(\beta)$ is the discrete spectrum of $A_{D,\pm}(\beta)$ then the problem $(\mathcal{P}^\pm)$ is well posed. Moreover, it exists $C$ independent from $\alpha$ such that
	\begin{equation}\label{eq:cont_u+}
		 \forall \varphi\in H^{1/2}_\beta(\Gamma_a^\pm), 
		\quad\norm{u^\pm(\beta,\alpha;\varphi)}_{H^1(B^\pm)}\leq \frac{C}{\text{dist}(\alpha^2,\sigma_\pm(\beta))} \norm{\varphi}_{H^{1/2}_\beta(\Gamma_a^\pm)}
	\end{equation}
	and so
	\[
			\forall \varphi\in H^{1/2}_\beta(\Gamma_a^\pm),\quad\norm{\Lambda^\pm(\beta,\alpha)\varphi}_{H^{-1/2}_{\beta}(\Gamma_a^\pm)}\leq \frac{C}{\text{dist}(\alpha^2,\sigma_\pm(\beta))}\norm{\varphi}_{H^{1/2}_\beta(\Gamma_a^\pm)}.
	\]
where $\sigma_\pm(\beta)$ is the spectrum of $A_{D,\pm}(\beta)$.\\\\
Let now $\alpha_1$ and $\alpha_2$ are such that each corresponding half-band problem $(\mathcal{P}^\pm)$ is well posed. For any $\beta$ and any data $\varphi$ in $H^{1/2}_\beta(\Gamma_a^\pm)$, $v^\pm = u^\pm(\beta,\alpha_1;\varphi)- u^\pm(\beta,\alpha_2;\varphi)$ is solution of 
\[
	\begin{array}{|l}
		-\triangle v^\pm-\rho_p\alpha_1^2 v^\pm = \rho_p(\alpha_2^2-\alpha_1^2) u^\pm(\beta,\alpha_2;\varphi)\quad\text{in}\;B^\pm \\[3pt]
		v^\pm|_{\Gamma_a^\pm} =0.
	\end{array}
\]
That implies that it exists $C$ independent from $\alpha_1$ and $\alpha_2$ such that
\[
	\norm{v}_{H^1_\beta(\Omega^\pm)}\leq \frac{C}{\text{dist}(\alpha_1^2,\sigma_\pm(\beta))}\abs{\alpha^2_1-\alpha^2_2}\norm{u^\pm(\beta,\alpha_2;\varphi)}_{H^1_\beta(B^\pm)}
\]
then using \eqref{eq:cont_u+}, we deduce that $\forall\varphi\in H^{1/2}_\beta(\Gamma_a^\pm),\quad$
\[
	\norm{\Lambda^\pm(\beta,\alpha_1)\varphi-\Lambda^\pm(\beta,\alpha_2)\varphi}_{H^{-1/2}_{\beta}(\Gamma_a^\pm)}\leq \frac{C}{\text{dist}(\alpha_1^2,\sigma_\pm(\beta))\,\text{dist}(\alpha_2^2,\sigma_\pm(\beta))}\abs{\alpha^2_1-\alpha^2_2}\norm{\varphi}_{H^{1/2}_\beta(\Gamma_a^\pm)}.
\]
\end{proof}~\\\\
Besides, the definition \ref{eq:def_dtn} allows us to establish the following property which will be useful for the proof of Lemma \ref{lem:garding_A0}.
\begin{lemma}\label{lem:garding_Lambda}
	For any $\beta$ and any $\alpha$ such that the problem $(\mathcal{P}^\pm)$ is well posed, the DtN operator $\Lambda^\pm(\beta,\alpha)$ satisfies the G{\aa}rding's inequality
	\begin{multline}
		\dsp\exists C_1,C(\alpha)>0,\quad\forall\varphi\in H^{1/2}_\beta(\Gamma^\pm_a)\\[5pt]
		\dsp	<\Lambda^\pm (\beta,\alpha)\varphi,\varphi>_{\Gamma^\pm_a}\; \geq C_1 \norm{\varphi}^2_{H^{1/2}_\beta(\Gamma_a^\pm)}-C(\alpha)\norm{\varphi}^2_{L^2(\Gamma_a^\pm)}
	\end{multline}
	where $C(\alpha)$ is a constant continuous with respect to $\alpha$.
\end{lemma}
\begin{proof}
	By definition of the DtN operator $\Lambda^\pm(\beta,\alpha)$ and by continuity of the trace application $\gamma_0^\pm$, we have
	\[
		\forall\varphi\in H^{1/2}_\beta(\Gamma^\pm_a),\quad<\Lambda^\pm (\beta,\alpha)\varphi,\varphi>_{\Gamma^\pm_a} \geq C_1 \norm{\varphi}^2_{H^{1/2}_\beta(\Gamma_a^\pm)}-(\alpha^2\rho_++1)\norm{u^\pm}^2_{L^2(B^\pm)}.
	\]
	where $u^\pm= u^\pm(\beta,\alpha;\varphi)$. To conclude, it is sufficient to show that 
	\[
		\forall\varphi\in H^{1/2}_\beta(\Gamma^\pm_a),\quad\norm{u^\pm}_{L^2(B^\pm)}\leq c(\alpha)\norm{\varphi}_{L^2(\Gamma_a^\pm)}
	\]
	with $c(\alpha)$ continuous with respect to $\alpha$. For this, we use a duality argument. Let $v^\pm\in H^1_\beta(\triangle,B^\pm)$ be the unique solution of the problem (which is well posed if the problem $(\mathcal{P}^\pm)$ is well posed)
	\begin{equation}\label{eq:v_pm}
		\begin{array}{|l}
			-\triangle v^\pm-\rho_p\alpha^2 v^\pm = u^\pm\quad\text{in}\;B^\pm \\[3pt]
			v|_{\Gamma_a^\pm} = 0.
		\end{array}
	\end{equation}
Since $u^\pm\in L^2(B^\pm)$ then $v^\pm\in H^2(B^\pm)$ (elliptic regularity) and using the well-posedness of the problem and the continuity of the trace aplication $\gamma_1^\pm$, we have in particular
\begin{equation}\label{eq:norm_v_pm}
	\norm{\frac{\partial v^\pm}{\partial x}}_{L^2(\Gamma_a^\pm)}\leq \frac{C}{\text{dist}(\alpha^2,\sigma_\pm(\beta))}\norm{u^\pm}_{L^2(B^\pm)}
\end{equation}
where $\sigma_\pm(\beta)$ is the spectrum of $A_{D,\pm}(\beta)$ defined in Theorem \ref{th:pb+-_bien}. Let us now multiply \eqref{eq:v_pm} by $u^\pm$, integrate over $B^\pm$ and apply Green's formula twice, we get
\[
	\norm{u^\pm}^2_{L^2(B^\pm)}=\pm\int_{\Gamma_a^\pm}\varphi\,\frac{\partial v^\pm}{\partial x}
\]
which together with \eqref{eq:norm_v_pm} finishes the proof.
\end{proof}
~\\\\
Finally, using the continuity of any eigenvector of \eqref{eq:eigpb} and its normal derivative across the sections $\Gamma^+_a$ and $\Gamma^-_a$, the next theorem is therefore straightforward.\\
\begin{theorem}[Problem with DtN conditions]\label{th:equiv_eigpb}
	The problem \eqref{eq:eigpb} is equivalent to the problem posed on $B_0=B\cap\Omega_0$
	\[
		\dsp \text{Find}\;\omega^2\notin \sigma_{ess}(\beta),\,\text{s.t.}\;\exists u_0 \in H^1(B_0),\;u_0\neq 0
	\]
\begin{equation}\label{eq:eigpb_0}\tag{$\mathcal{E}^\beta_0$}
\dsp	-\frac{1}{\rho}\triangle u_0=\omega^2 u_0,\quad\text{in}\;B_0
\end{equation}
$u_0$ satisfying the boundary conditions
\[
	 \dsp \begin{array}{|l}
	\dsp	+\partial_x\,u_0+\Lambda^+(\beta,\omega)\,u_0=0,\quad\text{on}\;\Gamma_a^+\\[5pt]
	\dsp-\partial_x\,u_0+\Lambda^-(\beta,\omega)\,u_0=0,\quad\text{on}\;\Gamma_a^-,\\[5pt]
	\dsp u_0|_{\Sigma_0}=e^{\imath\beta L_y}u_0|_{\widetilde{\Sigma}_0} ,\;\dsp \partial_y u_0|_{\Sigma_0}=e^{\imath\beta L_y}\partial_y u_0|_{\widetilde{\Sigma}_0}.
		\end{array}
\]
where $\Sigma_0 = \Sigma\cap\Omega_0 $ and $\widetilde{\Sigma}_0 = \widetilde{\Sigma}\cap\Omega_0 $.
\end{theorem}
\noindent These problems are equivalent in the sense that if $(\omega,u)$ is solution of \eqref{eq:eigpb} then $(\omega,u|_{B_0})$ is solution of \eqref{eq:eigpb_0}. Conversely, if $(u_0,\omega)$ is solution of \eqref{eq:eigpb_0} then $u$ defined by
	\[
		\begin{array}{|l}
			u|_{B_0} = u_0\\
			u|_{B^\pm} = u^\pm(\beta,\omega,\varphi),\quad\text{where}\;\varphi = u_0|_{\Gamma_a^\pm}
		\end{array}
	\]
	associated to the same value $\omega$ is solution of \eqref{eq:eigpb}. Moreover, the multiplicity of $\omega$ is the same for the two problems.
	\\\\
	Whereas the problem \eqref{eq:eigpb} was linear with respect to the eigenvalue $\omega^2$ but defined on an unbounded domain, the problem \eqref{eq:eigpb_0} is set on a bounded domain but non linear.\\
\begin{remark}
	Note that the problem \eqref{eq:eigpb_0} is also non linear with respect to $\beta$ (whereas the problem \eqref{eq:eigpb} can be rewritten as a quadratic eigenvalue problem). In other words, this difficulty would be present if we decided to fix $\omega$ and look for $\beta$.
\end{remark}~\\\\
\noindent We now  introduce the solution algorithm of the non linear eigenvalue problem, explain how to compute the DtN operators in the case where $\omega^2\notin\sigma_{ess}(\beta)$ and finally deduce another characterization of the essential spectrum $\sigma_{ess}(\beta)$.
	\subsection{Solution algorithm}\label{sub:solution_algorithm}
\noindent	For $\alpha^2\notin\sigma_{ess}(\beta)$, we denote by $A_0(\beta,\alpha)$ the operator 
	\[
	\begin{array}{|l}
		\dsp A_0(\beta,\alpha)=-\frac{1}{\rho_0}\triangle\\
		\dsp D(A_0(\beta,\alpha)) = \big\{u_0\in H^1(\triangle,B_0),\;
			 \dsp \begin{array}{|l}
			\dsp	+\partial_x\,u_0+\Lambda^+(\beta,\alpha)\,u_0=0,\quad\text{on}\;\Gamma_a^+\\[5pt]
			\dsp-\partial_x\,u_0+\Lambda^-(\beta,\alpha)\,u_0=0,\quad\text{on}\;\Gamma_a^-,\\[5pt]
			\dsp u_0|_{\Sigma_0}=e^{\imath\beta L_y}u_0|_{\widetilde{\Sigma}_0} ,\;\dsp \partial_y u_0|_{\Sigma_0}=e^{\imath\beta L_y}\partial_y u_0|_{\widetilde{\Sigma}_0}.
				\end{array}
		 \big\}.
	\end{array}
	\]
	The following lemma allows us to show the main properties of the operator $A_0(\beta,\alpha)$ given in Proposition \ref{prop:A_0_beta}
	\begin{lemma}\label{lem:garding_A0}
		The operator $A_0(\beta,\alpha)$ satisfies the G{\aa}rding's inequality 
		\begin{multline}
			\dsp\exists C_1,C(\alpha)>0,\quad\forall u \in D(A_0(\beta,\alpha))\\[5pt]
			\dsp	\left(A_0(\beta,\alpha)u,u\right)\; \geq C_1 \norm{u}^2_{H^1(B_0)}-C(\alpha)\norm{u}^2_{L^2(B_0)}
		\end{multline}
		where $C(\alpha)$ is a constant continuous with respect to $\alpha$.
	\end{lemma}
	\begin{proof}Using Lemma \ref{lem:garding_Lambda}, we obtain easily that it exists $c(\alpha)$, constant continuous with respect to $\alpha$, such that
		\[
			\forall u\in D(A_0(\beta,\alpha)),\quad \left(A_0(\beta,\alpha)u,u\right) \geq \norm{\nabla u}_{L^2(B_0)}^2 -c(\alpha)\left( \norm{u|_{\Gamma_a^+}}_{L^2(\Gamma_a^+)}^2+\norm{u|_{\Gamma_a^-}}^2_{L^2(\Gamma_a^-)}\right) .
		\]
	Let $s$ be in $]1/2,1[$. By continuity of the trace application in $H^s(B_0)$, we have
	\[
		\forall u\in H^s(B_0),\quad\norm{u|_{\Gamma_a^\pm}}^2_{L^2(\Gamma_a^\pm)}\leq C\norm{u}^2_{H^s(B_0)},
	\]	
	and from the compactness of the embedding $H^1(B_0) \hookrightarrow H^s(B_0)$, we could prove easily that
	\[
		\forall \varepsilon>0,\;\exists C_{\eps}>0,\;\forall u\in H^1(B_0),\quad \norm{u}_{H^s(B_0)}^2\leq \eps\norm{u}_{H^1(B_0)}^2+C_\eps \norm{u}_{L^2(B_0)}^2.
	\]
	Choosing $\eps$ small enough, these three relation give us the G{\aa}rding inequality. 
	\end{proof}~\\\\
	\begin{proposition}\label{prop:A_0_beta}
		The operator $A_0(\beta,\alpha)$ is self-adjoint, has a compact resolvant and is bounded from below. 
	\end{proposition}
\begin{proof} The only difficulty is to show the self-ajointness (the G{\aa}rding inequality implies that the operator is bounded from below).\\\\
	The operator $A_0(\beta,\alpha)$ is clearly symmetric. To show self-ajointness, it is sufficient to show that it exists a real number $\nu$ such that $A_0(\beta,\alpha)+\nu Id$ is surjective. The G{\aa}rding inequality gives us the coercivity and then the surjectivity of $A_0(\beta,\alpha)+C(\alpha) Id$.
\end{proof}
~\\\\The spectrum of $A_0(\beta,\alpha)$ is then a pure point one and consists of a sequence of eigenvalues $(\mu_n(\beta,\alpha))_n$  of finite multiplicity tending to $+\infty$:
\[
	\mu_1(\beta,\alpha)\leq \mu_2(\beta,\alpha)\leq\ldots\leq\mu_n(\beta,\alpha)\;\rightarrow\;+\infty
\]
The explicit expression of these eigenvalues using the Min-Max principle yields some regularity properties of each eigenvalue with respect to $\alpha$.\\
\begin{theorem}[Characterization of the eigenvalues $(\mu_n(\beta,\alpha))_n$]\label{th:charac_mum}
Let us fix $\beta \in [0,\pi/L_y[$ and $\alpha\notin \sigma_{ess}(\beta)$ and 
denote $H_0=L^2(B_0,\rho_0 dxdy)$. In the following $\mathcal{V}_m$ is the set of subspaces of dimension $m$ of 
\[	
H^1_\beta(B_0) = \left\{ u_0\in H^1(B_0),\;
\dsp\dsp u_0|_{\Sigma_0}=e^{\imath \beta}u_0|_{\widetilde{\Sigma}_0}
\right\}
\]
Then $\forall m\in\N,\quad$
\[
	\mu_m(\beta,\alpha) = \inf_{V_m\in\mathcal{V}_m}\,\sup_{u_0\in V_m,u_0\neq 0}\frac{\int_{B_0}\abs{\nabla u_0}^2+\left<\Lambda^+(\beta,\alpha)u_0,u_0\right>_{\Gamma_a^+}+\left<\Lambda^-(\beta,\alpha)u_0,u_0\right>_{\Gamma_a^
	-}}{\norm{u_0}_{H_0}^2}
\]
and $<\cdot,\cdot>_{\Gamma^\pm_a}$ denotes the duality product between $H^{-1/2}_{\beta}(\Gamma^\pm_a)$ and $H^{1/2}_{\beta}(\Gamma^\pm_a)$.\\\\
Moreover, the functions $(\beta,\alpha)\mapsto \mu_m(\beta,\alpha)$ are continuous.
\end{theorem}
\begin{proof}
	The first part of the theorem is a classical application of the min-max principle since $A_0(\beta,\alpha)$ is selfadjoint and bounded from below.\\\\	
	For the continuity, Lemma \ref{lem:garding_A0} gives us that for any $\beta$, it exists positive constants $C_1$ and $C(\alpha)$ which depends continuously of $\alpha$ such that for any $\alpha$
	\begin{equation}\label{eq:proof_coer}
		\forall u\in D(A_0(\beta,\alpha)),\quad \left(A_0(\beta,\alpha)u,u\right)+C(\alpha)\norm{u}_{L^2(B_0)}^2\geq C_1\norm{u}_{H^1(B_0)}^2
	\end{equation}
	Let us then consider the following Rayleigh quotients :
	\[
		\mathcal{R}(\beta,\alpha;u_0) = \frac{\int_{B_0}\abs{\nabla u_0}^2+\left<\Lambda^+(\beta,\alpha)u_0,u_0\right>_{\Gamma_a^+}+\left<\Lambda^-(\beta,\alpha)u_0,u_0\right>_{\Gamma_a^
		-} + C(\alpha)\norm{u_0}_{H^1(B_0)}}{\norm{u_0}^2_{H_0}}.
	\]	
	For any $\beta$, let $\alpha_1$ and $\alpha_2$ be such that the DtN operators $\Lambda^\pm(\beta,\alpha_i;\cdot)$ are well defined. For all $u_0\in H^1_\beta(B_0)$ such that $\norm{u_0}_{H_0}=1$, we have
	\begin{multline*}
		\abs{\mathcal{R}(\beta,\alpha_1;u_0)- \mathcal{R}(\beta,\alpha_2;u_0)}\leq \abs{\left<\left(\Lambda^+(\beta,\alpha_1)-\Lambda^+(\beta,\alpha_2)\right)u_0,u_0\right>_{\Gamma^+_a}} 
		\\  +\abs{\left<\left(\Lambda^-(\beta,\alpha_1)-\Lambda^-(\beta,\alpha_2)\right)u_0,u_0\right>_{\Gamma^-_a}}+\abs{C(\alpha_1)-C(\alpha_2)}\norm{u}^2_{H^1(B_0)}.
	\end{multline*}
So using Proposition \ref{prop:dtn_continuous} and its proof and the continuity of the trace application from $H^1(B_0)$ onto $H^{1/2}(\Gamma_a^\pm)$, we obtain the existence of a constant $C'$ independent from $\alpha_1$ and $\alpha_2$ such that
	\begin{multline*}
		\abs{\mathcal{R}(\beta,\alpha_1;u_0)- \mathcal{R}(\beta,\alpha_2;u_0)}\leq\dsp \frac{C'}{\text{dist}(\alpha_1^2,\sigma_+(\beta))\,\text{dist}(\alpha_2^2,\sigma_+(\beta))}\abs{\alpha^2_1-\alpha^2_2}\norm{u_0}_{H^1(B_0)}^2\\
		+ \abs{C(\alpha_1)-C(\alpha_2)}\norm{u_0}^2_{H^1(B_0)}.
	\end{multline*}
Taking into account finally the coercivity property \eqref{eq:proof_coer}, we obtain 
	\[
	\begin{split}
		\mathcal{R}(\beta,\alpha_2;u_0)\leq \mathcal{R}(\beta,\alpha_1;u_0)&+\dsp \frac{C'}{\text{dist}(\alpha_1^2,\sigma_+(\beta))\,\text{dist}(\alpha_2^2,\sigma_+(\beta))}\abs{\alpha^2_1-\alpha^2_2}\mathcal{R}(\beta,\alpha_1;u_0) \\&+ \abs{C(\alpha_1)-C(\alpha_2)}\mathcal{R}(\beta,\alpha_1;u_0).
	\end{split}
	\]
For all $m$, if we denote 
	\[
		\forall m\in\N,\quad\nu_m(\beta,\alpha) = \inf_{V_m\in\mathcal{V}_m}\,\sup_{u_0\in V_m,u_0\neq 0}\mathcal{R}(\beta,\alpha;u_0)
	\]
we obtain from the last inequality for all $m$
	\[
	\begin{split}
		\nu_m(\beta,\alpha_2)\leq \nu_m(\beta,\alpha_1)&+\dsp \frac{C'}{\text{dist}(\alpha_1^2,\sigma_+(\beta))\,\text{dist}(\alpha_2^2,\sigma_+(\beta))}\abs{\alpha^2_1-\alpha^2_2}\nu_m(\beta,\alpha_1)\\
		&+ \abs{C(\alpha_1)-C(\alpha_2)}\nu_m(\beta,\alpha_1).
	\end{split}
	\]
By exchanging the roles of $\alpha_1$ and $\alpha_2$, we finally obtain
	\[
	\begin{split}
	\abs{\nu_m(\beta,\alpha_2)-\nu_m(\beta,\alpha_1)} \leq &\dsp \frac{C'}{\text{dist}(\alpha_1^2,\sigma_+(\beta))\,\text{dist}(\alpha_2^2,\sigma_+(\beta))}\abs{\alpha^2_1-\alpha^2_2}\max\left(\nu_m(\beta,\alpha_1),\nu_m(\beta,\alpha_2)\right)
	\\&+ \abs{C(\alpha_1)-C(\alpha_2)}\;\max\left(\nu_m(\beta,\alpha_1),\nu_m(\beta,\alpha_2)\right).
	\end{split}
	\]
We deduce the continuity of the $\mu_m(\beta,\alpha)$, $C(\alpha)$ depending continuously of $\alpha$.
\end{proof}~\\\\
Consequently, using Theorem \ref{th:equiv_eigpb}, the solutions of the problem \eqref{eq:eigpb} are obtained as explained in the following proposition.
\begin{proposition}\label{prop:nonlinear_vp}
	The solutions of the non-linear problem \eqref{eq:eigpb} are the roots of the equations :
	\[	\omega^2\notin\sigma_{ess}(\beta)\quad\text{and}\quad\exists m\geq 1,\;\mu_m(\beta,\omega)=\omega^2.
	\]
\end{proposition}		
	We then infer the iterative algorithm for the computation of the guided modes and associated eigenvalues with two nested loops:
\begin{itemize}
\item the outer loop consists in a fixed point algorithm to solve the non linear equations:\\
$$\mu_m(\beta,\omega)=\omega^2,\quad\omega^2\notin\sigma_{ess}(\beta);$$
\item each iteration of this fixed point algorithm requires the computation of the $m-$th eigenvalue $\mu_m(\beta,\alpha)$ of the operator $A_0(\beta,\alpha)$ (and possibly the derivative of $\mu_m(\beta,\alpha)$ with respect to $\alpha$ if a Newton method is used to solve the fixed point problem). 
\end{itemize}
This algorithm is quite classical for the computation of guided modes in open waveguides (see \cite{Bonnet:2001}). Here the novelty comes from the facts that the eigenvalues $\omega^2$ could belong to any gap of the spectrum and that the operators $\Lambda^\pm(\beta,\omega)$ have no analytical expression. We show in the following section how they can be computed in practice.

	\subsection{Characterization of the DtN operators }
\noindent 	
As explained in \cite{Fliss:2006}, the construction of the operators $\Lambda^\pm(\beta,\alpha)$ is based only on the resolution of a family of cell problems and the resolution of a stationary Ricatti equation. For the sake of clarity, we shall recall the method for the construction of $\Lambda^+$. 
\\\\
We will suppose again in this section that $\alpha^2\notin\sigma_{ess}(\beta)$ and the problem $(\mathcal{P}^+)$ is well posed.
\\\\
We shall use the division of the half-band into periodicity cells separated by vertical segments
(See Figure~\ref{fig:notation_2dw}):
\begin{equation}\label{eq:def_waveguide}
	B^+ = \bigcup_{n=1}^{+\infty} \; \calc_n, \quad \calc_n := \calc_1 +((n-1)L_x,0),
\end{equation}
The segments $\Gamma_n = \Gamma_0+(nL,0)$, with $\Gamma_0=\Gamma^+_a$ can all be identified to the leftmost one $\Gamma_0\sim [-L_y/2,L_y/2]$ and the cells $\calc_n$ can all be identified to the first cell $\calc_1=\calc$.
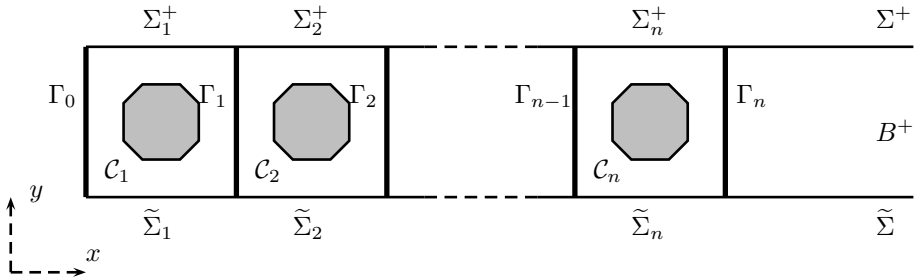
\begin{figure}[htbp]
\begin{center}
	\psset{unit=0.5cm} \psset{linewidth=1pt} 
	\begin{pspicture}(-4,-4)(30,30)
	\psline[linestyle=dashed]{->}(3,-1)(5,-1)
	\psline[linestyle=dashed]{->}(3,-1)(3,1) \put(5,-0.7){$x$}\put(3.5,1){$y$}
	\psline[linewidth=2pt](5,1)(5,5)\put(4,3.5){$\Gamma_0$}
	\put(5.5,1.5){$\calc_1$}
	\pspolygon[fillstyle=solid,fillcolor=lightgray](6.5,2)(7.5,2)(8,2.5)(8,3.5)(7.5,4)(6.5,4)(6,3.5)(6,2.5)
	\psline[linewidth=2pt](9,1)(9,5)\put(8,3.5){$\Gamma_1$}
	\put(9.5,1.5){$\calc_2$}
	\pspolygon[fillstyle=solid,fillcolor=lightgray](10.5,2)(11.5,2)(12,2.5)(12,3.5)(11.5,4)(10.5,4)(10,3.5)(10,2.5)
	\psline[linewidth=2pt](13,1)(13,5)\put(12,3.5){$\Gamma_2$}
	\psline[linewidth=2pt](18,1)(18,5)\put(16.4,3.5){$\Gamma_{n-1}$}
	\put(18.5,1.5){$\calc_n$}
	\pspolygon[fillstyle=solid,fillcolor=lightgray](19.5,2)(20.5,2)(21,2.5)(21,3.5)(20.5,4)(19.5,4)(19,3.5)(19,2.5)
	\psline[linewidth=2pt](22,1)(22,5)\put(22.3,3.5){$\Gamma_n$}
	\psline(5,1)(14,1)\psline[linestyle=dashed](14,1)(17,1)\psline(17,1)(27,1)\put(26,0){$\widetilde{\Sigma}$}\put(6.5,0){$\widetilde{\Sigma}_1$}\put(10.5,0){$\widetilde{\Sigma}_2$}\put(19.5,0){$\widetilde{\Sigma}_n$}
	\psline(5,5)(14,5)\psline[linestyle=dashed](14,5)(17,5)\psline(17,5)(27,5)\put(26,5.5){$\Sigma^+$}\put(6.5,5.5){$\Sigma^+_1$}\put(10.5,5.5){$\Sigma^+_2$}\put(19.5,5.5){$\Sigma^+_n$}
	\put(26,2.5){$B^+$}
	\end{pspicture}
\caption{Notation in a half band}\label{fig:notation_2dw}
\end{center}
\end{figure}
By periodicity in $x$, the construction of the unique solution $u^+(\beta,\alpha,\varphi)$ of $(\mathcal{P}^+)$ will be reduced to the knowledge of a linear operator, called the propagation operator, denoted $P(\beta,\alpha)$ and defined by
\begin{equation}\label{eq:op_P_betaalpha}
	\begin{array}{rl}
	P(\beta,\alpha):\:H^{1/2}_{\beta}(\Gamma_0) \rightarrow& H^{1/2}_{\beta}(\Gamma_0)\\[1pt]
	\:\varphi\mapsto &\dsp u^+(\beta,\alpha,\varphi)\Big|_{\Gamma_1}.
	\end{array}
\end{equation}
One can show that $P(\beta,\alpha)$ for any $\beta$ and $\alpha$ is compact (using interior elliptic regularity for $u^+(\beta,\alpha,\varphi)$), injective  (using an argument of unique continuation) and has a spectral radius less than 1 when $\alpha^2\notin\sigma_{ess}(\beta)$ (because of the $L^2$ nature of $u^+(\beta,\alpha,\varphi)$). See \cite{Fliss:2006,Fliss:2009} for more details on the proof of these results.
\\\\
Using the periodicity of the problem, one easily see that
\[
	\forall n\in\N,\;\forall \varphi\in H^{1/2}_{\beta}(\Gamma_0),\quad u^+(\beta,\alpha,\varphi)\Big|_{\Gamma_n} = \left(	P(\beta,\alpha)\right)^n\,\varphi.
\]
Then by linearity, we have $\forall n\in\N,\;\forall \varphi\in H^{1/2}_{\beta}(\Gamma_0)$
\begin{equation}\label{eq:rebuild_u}
	u^+(\beta,\alpha,\varphi)\Big|_{\mathcal{C}_n} = e^0\left(\beta,\alpha,\left(	P(\beta,\alpha)\right)^{n-1}\,\varphi\right) + e^1\left(\beta,\alpha,\left(	P(\beta,\alpha)\right)^{n}\,\varphi\right)
\end{equation}
where for all $\varphi$ in $H^{1/2}_{\beta}(\Gamma_0)$ $e_0(\beta,\alpha,\varphi)$ and $e_1(\beta,\alpha,\varphi)$ are the solutions in $H^1(\calc)$ to the following elementary cell problems
	\begin{equation}\label{eq:cellpb}
		-\triangle e_\ell-\rho_p\,\alpha^2e_\ell = 0\quad\text{in}\;\mathcal{C},\quad\quad\ell\in\{0,1\}
	\end{equation}
satisfying $\beta$-quasi periodic boundary conditions on $\Sigma_1=\Sigma\cap\calc$ and $\widetilde{\Sigma}_1=\widetilde{\Sigma}\cap\calc$
	\begin{equation}\label{eq:cellpbQP}
		\dsp e_\ell|_{\Sigma_1}=e^{\imath\beta L_y}\,e_\ell|_{\widetilde{\Sigma}_1} \quad\quad
		\dsp \partial_y e_\ell|_{\Sigma_1}=e^{\imath\beta L_y}\partial_y \,e_\ell|_{\widetilde{\Sigma_1}}
	\end{equation}
	and Dirichlet boundary conditions on $\Gamma^0$ and $\Gamma^1$
	\begin{equation}\label{eq:cellpbcond}
	\begin{array}{ccc}
		\begin{array}{|l}
		e_0|_{\Gamma^0} = \varphi\\[3pt]
		e_0|_{\Gamma^1} = 0
		\end{array}&\text{and}&
	\begin{array}{|l}
	e_1|_{\Gamma^0} = 0\\[3pt]
	e_1|_{\Gamma^1} = \varphi
	\end{array}
	\end{array}.
	\end{equation}
	\begin{figure}[htbp]
			\begin{center}
				\psset{unit=0.5cm} \psset{linewidth=1pt} 
				\begin{pspicture}(-5,-5)(30,30)
					\pspolygon[fillstyle = solid,fillcolor=gray](3,0.5)(4,0.5)(4.5,1)(4.5,2)(4,2.5)(3,2.5)(2.5,2)(2.5,1)\pspolygon(2,0)(2,3)(5,3)(5,0)\psline[linewidth = 3pt](2,0)(2,3)\psline[linewidth = 3pt](5,0)(5,3)\put(3,0.8){$\calc$}\put(1.1,0.1){$\Gamma^0$}\put(5.2,0.1){$\Gamma^1$}\put(3.2,3.3){$\Sigma_1$}\put(3.2,-0.6){$\widetilde{\Sigma}_1$}			
				\end{pspicture}
		\caption{The periodicity cell $\calc$}\label{fig:sol_e0e1_2dw_k} 
			\end{center}
		\end{figure}
\begin{remark}
The cell problems are well posed except for a countable set of frequencies which correspond to the eigenvalues of the operator $-\triangle/\rho_p$ defined in a cell with Dirichlet boundary conditions. This set of forbidden frequencies is {\it a priori} different from the one introduced in Theorem \ref{th:pb+-_bien} and Remark \ref{rem:RtR_vs_DtN}. Indeed the first one is introduced because of the Dirichlet boundary conditions set in the vertical boundaries of the cell whereas the second one is introduced because of the Dirichlet boundary conditions set in the vertical boundary of the half band. Imposing Robin-type boundary conditions instead of Dirichlet in the cell problems would solve this issue (and can be used for the characterization of DtN or Robin-to-Robin operators). We choose here Dirichlet-type cell problems only to simplify the presentation of the method.
\end{remark}
~\\\\Formula \eqref{eq:rebuild_u} shows that the computation of the solution $u^+$ is achieved through the characterization of the operator $P(\beta,\alpha)$. At this stage of the exposition, the definition of this operator relies on $u^+(\beta,\alpha,\varphi)$ which is a solution of a problem posed in an unbounded domain. We shall see in the following how to determine this operator by solely solving local  problems of the type (\ref{eq:cellpb}, \ref{eq:cellpbQP}, \ref{eq:cellpbcond}), which is one key point of the method.
\\\\
Note that the relation \eqref{eq:rebuild_u} ensures that $u^+(\beta,\alpha,\varphi)$ is the solution of the Helmholtz equation inside each cell $\calc_n$. To make the characterization complete, we have to add the transmission condition across  $\Gamma_j$ (the continuity of the normal derivative of $u^+$ accross $\Gamma_j$), which gives
\begin{equation*}	T_{10}(\beta,\alpha)P(\beta,\alpha)^2+(T_{00}(\beta,\alpha)+T_{11}(\beta,\alpha))\,P(\beta,\alpha)+T_{01}(\beta,\alpha) = 0.
\end{equation*}
where the operators $T_{pq}(\beta,\alpha)$ are four corresponding local DtN operators :
	\begin{equation}\label{eq:dtn_local}
	\begin{array}{ll}
		 T_{00}(\beta)\varphi= - \partial_xe_0(\beta,\varphi)\Big|_{\Gamma^0}&T_{10}(\beta)\varphi= - \partial_xe_1(\beta,\varphi)\Big|_{\Gamma^0}\\[5pt]
		 T_{01}(\beta)\varphi= - \partial_xe_0(\beta,\varphi)\Big|_{\Gamma^1}&T_{10}(\beta)\varphi= - \partial_xe_1(\beta,\varphi)\Big|_{\Gamma^1}
		\end{array}
	\end{equation}
Actually,  for any $\beta$, if $\alpha^2\notin\sigma_{ess}(\beta)$, this equation characterizes uniquely the operator $P(\beta,\alpha)$:
	\begin{theorem}[Characteristic equation]\label{th:eqcarac}
Suppose that $\alpha^2\notin\sigma_{ess}(\beta)$ and the problem $(\mathcal{P}^+)$ is well posed. The operator $P(\beta,\alpha)$ is then the unique compact operator 
	of $\mathcal{L}\big(H^{1/2}_{\beta}(\Gamma_0)\big)$ satisfying the condition
		\begin{equation}\label{eq:cond_P_beta}
			\rho(P(\beta,\alpha))<1
		\end{equation}
which solves the stationary Riccati equation:		\begin{equation}\label{eq:eqcarac_P_beta}\tag{$\mathcal{E}^P(\beta,\alpha)$}
				\mathcal{T}(\beta,\alpha)\,X =0,
		\end{equation}
where
		\begin{equation*}\label{eq:mathcalT_beta}
			\begin{array}{rl}				\mathcal{T}(\beta,\alpha)\;:\;\mathcal{L}\big(H^{1/2}_{\beta}(\Gamma_0)\big) \rightarrow&  \mathcal{L}\big( H^{1/2}_{\beta}(\Gamma_0)\big)\\[5pt]
				X\mapsto&T_{10}(\beta,\alpha)X^2+(T_{00}(\beta,\alpha)+T_{11}(\beta,\alpha))\,X+T_{01}(\beta,\alpha).
			\end{array}
		\end{equation*}
	\end{theorem}
	\begin{proof}
		As explained before, the operator $P(\beta,\alpha)$ is a compact operator whose spectral radius is strictly less than one. Moreover, the continuity of the normal derivative of the unique solution $u^+(\beta,\alpha;\varphi)$ accross each $\Gamma_j$, $P(\beta,\alpha)$ is solution of \eqref{eq:eqcarac_P_beta}.\\\\
		Reciprocally, let $P$ be a compact operator with spectral radius $\rho(P)<1$ which satisfies \eqref{eq:eqcarac_P_beta}. The function defined by
		\[
		u(\varphi)\Big|_{\mathcal{C}_n} = e^0\left(\beta,\alpha,	P^{n-1}\,\varphi\right) + e^1\left(\beta,\alpha,	P^{n}\,\varphi\right)
		\]
		satisfies the helmoltz equation cell by cell, is continuous by definition and has continuous normal derivatives accross each $\Gamma_j$ because the operator $P$ is solution of \eqref{eq:eqcarac_P_beta}. Moreover, the fact that $\rho(R)<1$ implies that $u(\varphi)$ belongs to $H^1(B^+)$. Indeed, the property
		\[
			\lim_{n\rightarrow+\infty}\norm{P^n}^{1/n}_{\mathcal{L}(H^{1/2}(\Gamma))}=\rho(P)
		\]
		implies that, for some $\rho_*\in(\rho(P),1)$ and $N$ large enough
		\[
			\forall n\geq N,\quad\norm{P^n}_{\mathcal{L}(H^{1/2}(\Gamma))}\leq \rho_*^n
		\]
		so that, it exists $C$ independant of N, such that
		\[
			\forall n\geq N,\quad\norm{u(\varphi)}_{H^1(\mathcal{C}_n)}\leq C\rho_*^n\norm{\varphi}_{H^{1/2}_{\beta}(\Gamma_0)}.
		\]
		We have then constructed a solution in $H^1_\beta(\triangle,B^+)$ of $(\mathcal{P}^+)$. By well-posedness of this problem, the solution is unique so $u=u^+$ and then $P=P(\beta,\alpha)$.
	\end{proof}~\\\\
\noindent Once $P(\beta,\alpha)$ is determined solving the stationary Ricatti equation, we build cell by cell the solution $u^+$ using \eqref{eq:rebuild_u} and finally using again \eqref{eq:rebuild_u} for $n=0$, we see that
	\begin{equation}\label{eq:expression_dtn}
		\Lambda^+(\beta,\alpha) = T_{00}(\beta,\alpha) + T_{10}(\beta,\alpha)\,P(\beta,\alpha)
	\end{equation}
		\subsection{Characterization of the essential spectrum}\label{sec:ess_spectrum}
	\noindent The characterization of the essential spectrum is a by-product of the determination of the DtN operators.\\\\ Indeed, in \cite{Fliss:2006}, we have shown that when $\alpha^2$ lies in the essential spectrum $\sigma_{ess}(\beta)$, not only the spectral radius of the propagation operator is equal to 1 but in addition this operator is no more the unique solution of the Ricatti equation \eqref{eq:eqcarac_P_beta}. Consequently, we can give this characterization of the essential spectrum of $A$.\\
\begin{theorem}\label{th:ess_spec}
The Ricatti equation \eqref{eq:eqcarac_P_beta} has a unique solution whose spectral radius is strictly less than one if and only if $\alpha^2\notin\sigma_{ess}(\beta)$.
		\end{theorem}~\\
\begin{proof}
	In Theorem \ref{th:eqcarac}, we have already proven that if $\alpha^2\notin\sigma_{ess}(\beta)$, the Ricatti equation \eqref{eq:eqcarac_P_beta} has a unique solution whose spectral radius is strictly less than one.\\\\
	For $\alpha^2\in\sigma_{ess}(\beta)$, we can show that it does not exist solution of \eqref{eq:eqcarac_P_beta} with spectral radius strictly less than one but solution whose spectral radius equal to one. Moreover, there is no more uniqueness of this kind of solutions. Indeed, by definition of the essential spectrum (see the proof of Proposition \ref{prop:essential_spectrum})
	\[
		\alpha^2\notin\sigma_{ess}(\beta)\quad\Leftrightarrow \quad\exists q,\exists k_0\in]-\pi/L_x,\pi/L_x]\quad\alpha^2=\omega_q(\beta,k_0)
	\]
	where $\omega_q(\beta,k)$ is the $q-$th eigenvalue of the operator $A_p(\beta,k)$ defined in the proof of Proposition \ref{prop:essential_spectrum}. Let us denote $e_q(\beta,k)$ a corresponding eigenvector.
	Using the same idea than in the proof of Proposition \ref{prop:dtn_continuous}, we could show that $k\mapsto\omega_q(\beta,k)$ is an even function and then
	\[
		\omega_q(\beta,k_0) = \omega_q(\beta,-k_0) = \alpha^2
	\]
	Let us denote now $\varphi^\pm=e_q(\beta,\pm k_0)$. Using the well-posedness of the cell problems,  the eigenvectors $e_q(\beta,\pm k_0)$ can be expressed by
	\[
		e_q(\beta,\pm k_0)=e^0(\beta,\alpha,\varphi^\pm)+e^{\pm\imath k_0L_x}\,e^1(\beta,\alpha,\varphi^\pm)
	\]
	and, thanks to their quasi-periodicity, can be extended as $H^1(\triangle,B^+)$ function by
	\[
		e_q(\beta,\pm k_0)\Big|_{\mathcal{C}_n}=e^{\pm\imath(n-1) k_0L_x}\,e^0(\beta,\alpha,\varphi^\pm)+e^{\pm\imath n k_0L_x}\,e^1(\beta,\alpha,\varphi^\pm).
	\]
	We have found then solutions of the problem $(\mathcal{P}^+)$ with $\varphi^\pm$ as Dirichlet boundary condition on $\Gamma_0$. We can then construct at least three solutions of the Ricatti equation whose spectral radius is equal to one :
	\[
		\begin{array}{|ll}
			P_1&\;\text{such that}\; P_1\,\varphi^+=e^{\imath k_0L_x}\,\varphi^+\\[5pt]
			P_2&\;\text{such that}\; P_2\,\varphi^-=e^{-\imath k_0L_x}\,\varphi^-\\[5pt]
			P_3&\;\text{such that}\; P_3\,\varphi^+=e^{\imath k_0L_x}\,\varphi^+\;\text{and}\;P_3\,\varphi^-=e^{-\imath k_0L_x}\,\varphi^-
		\end{array}
	\]
\end{proof}
\section{Numerical results}
\noindent 
We want to apply this approach to compute the guided modes of the media, whose 
refraction index is given by (see the isovalues Figure~\ref{fig:media}): 
\begin{equation}\label{eq:indice}
	\forall\,(x,y)\in[-0.5,0.5]^2,\forall\, n,\, m\in\Z^*,\;\begin{array}{|l}
	\dsp\rho_p(x+n,y+m) = 1+16\exp(-\frac{x^2+y^2}{0.2^2})\\
	\dsp\rho_0(x,y+m)=1
\end{array}
\end{equation}
The period here is equal to $1$.
\begin{figure}[htbp]
	\begin{center}
	\includegraphics[width=8cm]{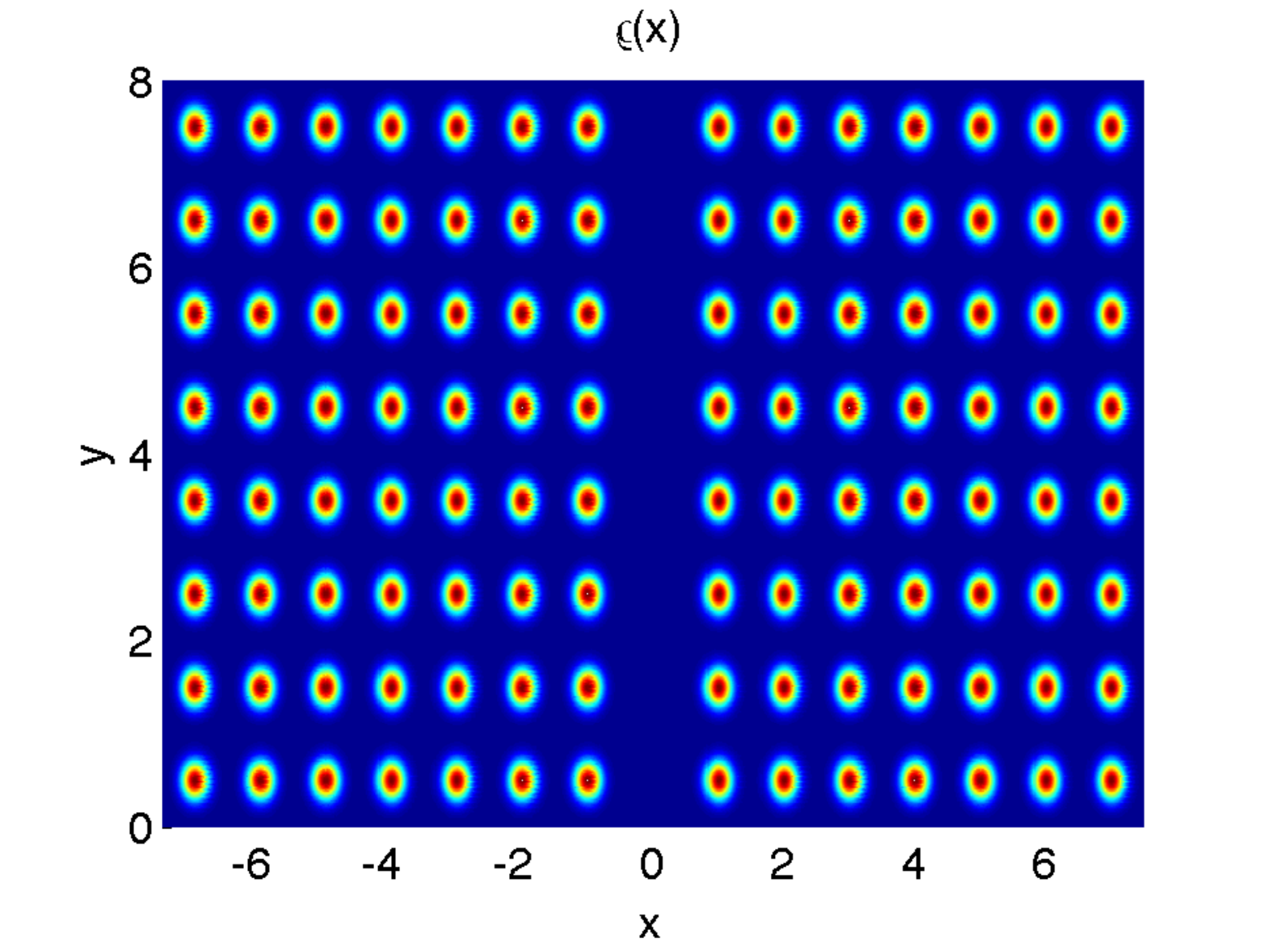}
		\caption{The lineic perturbed periodic media.}\label{fig:media}
	\end{center}
\end{figure}
~\\\\
For any $\beta\in[0,\pi/L_y]$ and $\alpha^2\in[0,20]$, we determine if $\alpha^2$ is in $\sigma_{ess}(\beta)$ or not. If not, we compute the Dirichlet-to-Neumann operators $\Lambda^\pm(\beta,\alpha)$. The corresponding algorithm is the following : for each $\beta$ and $\alpha$
\begin{enumerate}
	\item solve the cell problems (\ref{eq:cellpb},\ref{eq:cellpbQP},\ref{eq:cellpbcond}) - using a FE method or a mixed finite element one (see \cite{Fliss:2006,Fliss:2009,Fliss:2011} for more details) and compute the local DtN operators $T_{pq}^\pm(\beta,\alpha)$ (see \eqref{eq:dtn_local});
	\item solve the stationary equation \eqref{th:eqcarac} using a spectral decomposition method (which leads to a quadratic eigenvalue problem) or a modified Newton algorithm (see \cite{Fliss:2006,Fliss:2009,Fliss:2011} for more details):
	\begin{itemize}
		\item[(i)] If it exists a unique solution whose spectral radius is strictly less than one then $\alpha^2\notin\sigma_{ess}(\beta)$, so we can determine the propagation operator $P^\pm(\beta,\alpha)$ and deduce the DtN operators $\Lambda^\pm(\beta,\alpha)$ by \eqref{eq:expression_dtn}.
		\item[(ii)] If it exists a solution whose spectral radius is equal to one then $\alpha^2\in\sigma_{ess}(\beta)$ and we stop the computation.
	\end{itemize}
\end{enumerate}
We can then, for any $\beta$ and $\alpha$, compute the eigenvalues $\mu_m(\beta,\alpha)$.
\\\\
We plot on Figure~\ref{fig:lignesniveau_mu1} the isovalue lines of the function $\log\abs{\mu_1(\beta,\alpha)-\alpha^2}$ for $\beta\in[0,\pi/L_y]$ and $\alpha^2\in[0,20]\setminus\sigma_{ess}(\beta)$.
\begin{figure}[htbp]
	\begin{center}
		\includegraphics[width=12cm]{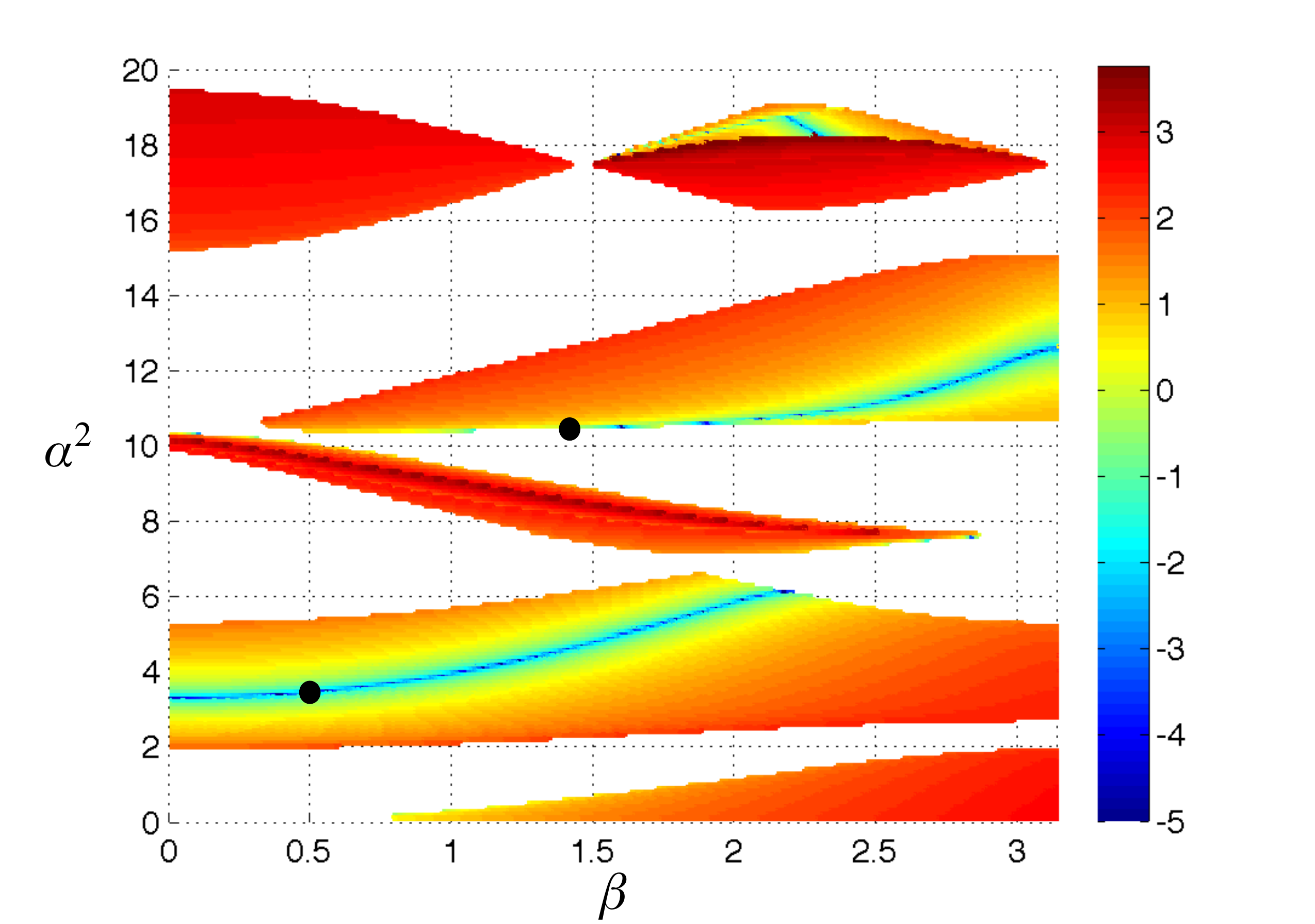}
		\caption{Isovalues of $\log\abs{\mu_1(\beta,\alpha)-\alpha^2}$; $\beta\in[0,\pi/L]$, $\alpha^2\in[0,20]$. For a fixed $\beta$, the white regions correspond to the essential spectrum $\sigma_{ess}(\beta)$. The dispersion curves are given by the blue lines. The two black circles correspond to guided modes represented Figure~\ref{fig:modes}.}\label{fig:lignesniveau_mu1}
	\end{center}
\end{figure}
~\\
For a fixed $\beta$, the white regions correspond to the essential spectrum $\sigma_{ess}(\beta)$. A part of the dispersion curves, given by the values $(\beta,\omega)$ for which $\mu_1(\beta,\omega)-\omega^2 = 0$, are represented by blue lines corresponding to the null isovalue of $\mu_1(\beta,\alpha)-\alpha^2$. The main difference with homogeneous open waveguides is that the
function $\mu_m(\beta,\alpha)-\alpha^2$ may vanish several times for a fixed $\beta$. 
\\\\
We want now to compute for fixed $\beta$, the eigenvalues and the corresponding guided modes. To do so, we use the following steps :
\begin{enumerate}
	\item {\bf Solution of the non linear eigenvalue problem :} using the algorithm of the non linear eigenvalue problem described at the end of Section \ref{sub:solution_algorithm}, $\omega^2$ is computed so that it exists $m$ such that $\mu_m(\beta,\omega)=\omega^2$ and $\omega^2\notin\sigma_{ess}(\beta)$. Let $u_0$ be an eigenvector of $A_0(\beta,\omega)$ associated to the eigenvalue $\mu_m(\beta,\omega)$.
	\item {\bf Construction of the guided modes in one band :} by continuity arguments, the restriction to a band $B$ of a guided mode $u$ is given by
	\[
		\begin{array}{|lcl}
			u\big|_{B_0} &=& u_0\\[5pt]
			u\big|_{B^+} &=& u^+(\beta,\omega,\varphi^+)\quad\text{with}\;\varphi^+ = u_0\big|_{\Gamma^+_a}\\[5pt]
			u\big|_{B^-} &=& u^-(\beta,\omega,\varphi^-)\quad\text{with}\;\varphi^- = u_0\big|_{\Gamma^-_a}
		\end{array}
	\]
	where for any given $\varphi\in H^{1/2}_\beta(\Gamma_a^\pm)$, $u^\pm(\beta,\omega,\varphi)$ is the unique solution of \eqref{eq:pbplus}. To compute these solutions, we have to
	\begin{itemize}
	\item solve the cell problems (\ref{eq:cellpb},\ref{eq:cellpbQP},\ref{eq:cellpbcond}) for $\alpha=\omega$ and compute the local DtN operators $T_{pq}^\pm(\beta,\omega)$ (see \eqref{eq:dtn_local});
	\item solve the stationary equation \eqref{th:eqcarac}. Since $\omega^2\notin\sigma_{ess}(\beta)$, it exists a unique solution whose spectral radius is strictly less than one which is  the propagation operator $P^\pm(\beta,\omega)$ 
	\item reconstruct the solutions cell by cell using \eqref{eq:rebuild_u}.
\end{itemize}
\item {\bf Construction of the guided modes in the whole domain :} by $\beta-$quasi-periodicity, the guided modes can be constructed "band by band" :
\[
	\forall(x,y)\in B,\;\forall q\in\Z,\quad u(x,y+qL_y) = u(x,y)e^{\imath q \beta L_y}.
\]
\end{enumerate}
We then represent  Figure~\ref{fig:modes} guided modes for two values of $\beta$ in eight periods from each side of the line defect. Figure ~\ref{subfig:modeconfine} corresponds to the eigenvalue in the first gap of $A(0.5)$ with a well confined mode whereas in Figure~\ref{subfig:modenonconfine} the eigenvalue belongs to the fourth gap of $A(1.9)$) and the associated guided mode is not well confined. 
\begin{figure}[htbp]
	\begin{center}
		\subfigure[$\beta=0.5,\,\omega^2 = 3.465$]{	\includegraphics[width=8cm]{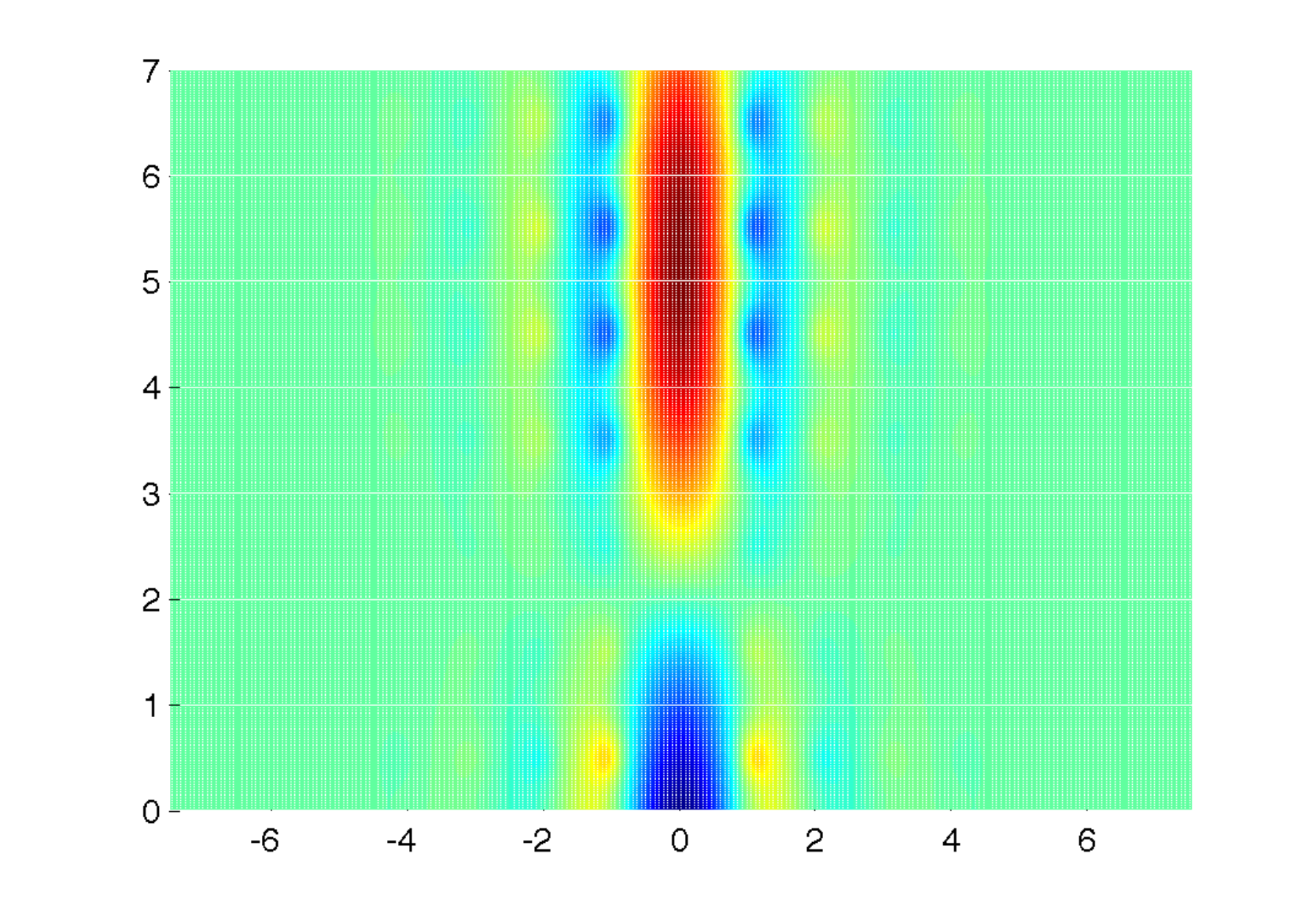}\label{subfig:modeconfine}
		}
		\subfigure[$\beta=1.42,\,\omega^2 = 10.46$]{
	\includegraphics[width=8cm]{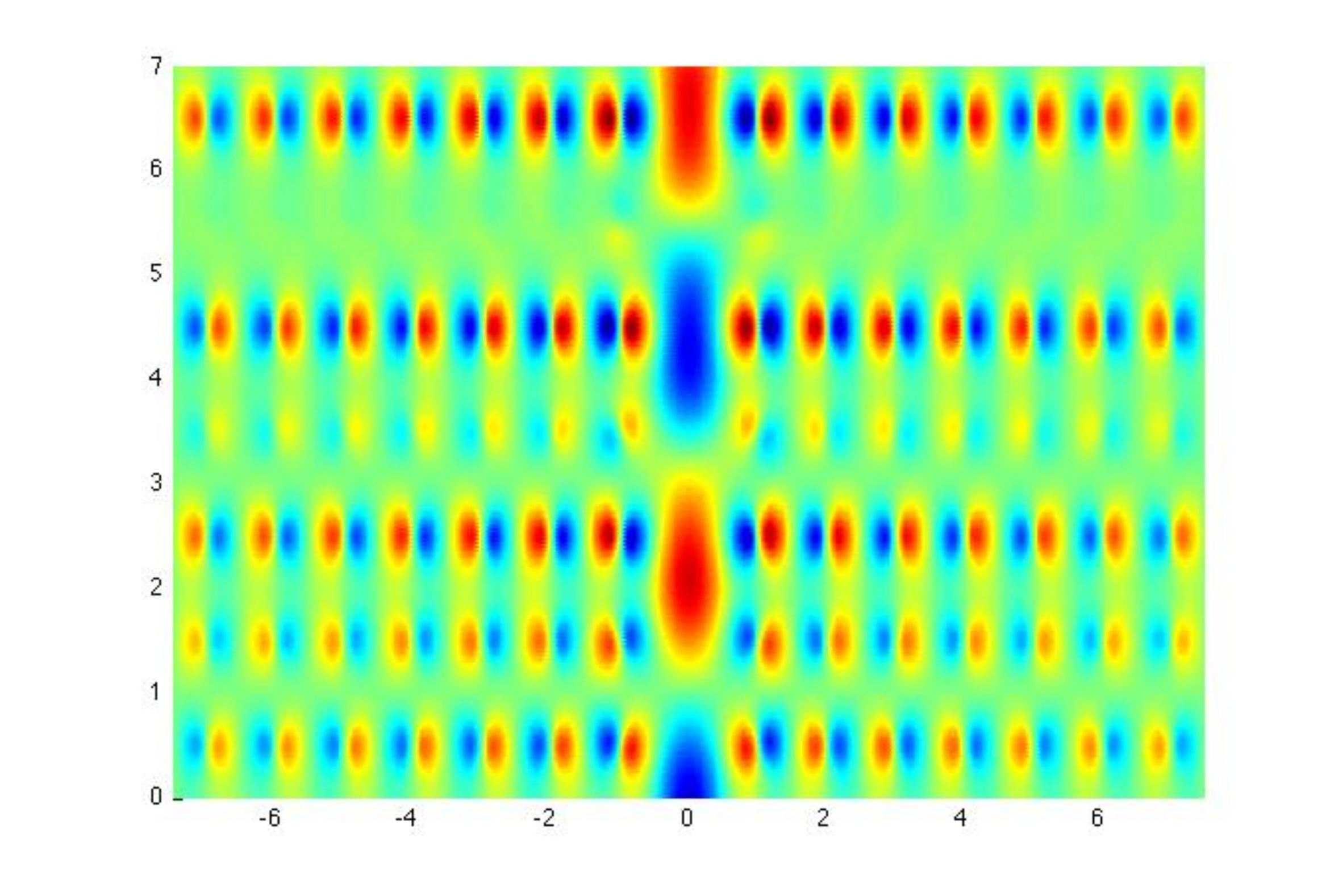}\label{subfig:modenonconfine}
		}
		\caption{Two guided modes: one is well confined (left); the other not well confined (right). The couples $(\beta,\omega^2)$ of this guided modes are represented by black circles in Figure~\ref{fig:lignesniveau_mu1}}\label{fig:modes}
	\end{center}
\end{figure}
\section{Conclusions}
\noindent We have shown that this method is well adapted to the computation of guided modes which are not well confined. The advantage of this method is that it is exact independently from the confinement of the computed guided modes. It can be well adapted for the study of the dispersive curves near the essential spectrum. In terms of computation time, the solution of the non linear eigenvalue problem is much more costly than the one of a linear eigenvalue problem. However, in contrast to existing methods, we do not have to compute independently -and often using a separate code- the essential spectrum of the media: here it is a by-product of this method.\\\\
In a forthcoming article, we will compare exactly this method with the Supercell Method in terms of accuracy and computation time. Moreover, this DtN strategy will be applied to a $\beta$-formulation  which is more adapted for dispersive media.\\\\
Another interesting perspective concerns the study of the dispersive curves: with our characterization, we could study if they can or cannot be constant with respect to $\beta$ (and conclude on the existence or not of bound state) and analyze their behavior near the essential spectrum.

\section*{Acknowledgments}
The author would like to thank P. Joly and A.-S. Bonnet-BenDhia for numerous discussions on the subject of this article.

\bibliographystyle{siam}
\bibliography{biblio}

\begin{thebibliography}{10}

\bibitem{Ammari:2004}
{\sc H.~Ammari and F.~Santosa}, {\em {Guided waves in a photonic bandgap
  structure with a line defect}}, SIAM J. Appl. Math., 64 (2004),
  pp.~2018--2033 (electronic).

\bibitem{Bonnet:2001}
{\sc A.-S. Bonnet-BenDhia and P.~Joly}, {\em Mathematical analysis and
  numerical approximation of optical waveguides.}, vol.~22 of Frontiers Appl.
  Math., SIAM, Philadelphia, PA, 2001, pp.~273--324.

\bibitem{Bonnet:1994}
{\sc A.-S. Bonnet-BenDhia and F.~Starling}, {\em Guided waves by
  electromagnetic gratings and non-uniqueness examples for the diffraction
  problem}, Math. Methods Appl. Sci., 17 (1994), pp.~305--338.

\bibitem{Borg:1946}
{\sc G.~Borg}, {\em Eine {U}mkehrung der {S}turm-{L}iouvilleschen
  {E}igenwertaufgabe. {B}estimmung der {D}ifferentialgleichung durch die
  {E}igenwerte}, Acta Math., 78 (1946), pp.~1--96.

\bibitem{Botten:2006}
{\sc L.C. Botten, K.B. Dossou, S.~Wilcox, R.C. McPhedran, C.M. de~Sterke, N.A.
  Nicorovici, and A.A. Asatryan}, Int. J. Microwave Opt. Technol.

\bibitem{Cances:2012}
{\sc E.~Canc\'es, V.~Ehrlacher, and Y.~Maday}, {\em {Periodic Schr\"odinger
  operators with local defects and spectral pollution}}, submitted to SIAM J.
  Numer. Anal.,  (2012).

\bibitem{Cardone:2009}
{\sc G.~Cardone, S.~A. Nazarov, and C.~Perugia}, {\em {A gap in the continuous
  spectrum of a cylindrical waveguide with a periodic perturbation of the
  surface}}, arXiv.org, math.SP (2009).

\bibitem{Figotin:1997a}
{\sc A.~Figotin and Y.~A. Godin}, {\em The computation of spectra of some
  {$2$}{D} photonic crystals}, J. Comput. Phys., 136 (1997), pp.~585--598.

\bibitem{Figotin:1996c}
{\sc A.~Figotin and A.~Klein}, {\em {Localization of classical waves .1.
  Acoustic waves}}, Comm. Math. Phys., 180 (1996), pp.~439--482.

\bibitem{Figotin:1997}
\leavevmode\vrule height 2pt depth -1.6pt width 23pt, {\em Localized classical
  waves created by defects}, J. Statist. Phys., 86 (1997), pp.~165--177.

\bibitem{Figotin:1998a}
\leavevmode\vrule height 2pt depth -1.6pt width 23pt, {\em Localization of
  light in lossless inhomogeneous dielectrics}, J. Opt. Soc. Am., A 15 (1998),
  pp.~1423--1435.

\bibitem{Figotin:1998b}
\leavevmode\vrule height 2pt depth -1.6pt width 23pt, {\em Midgap defect modes
  in dielectric and acoustic media}, SIAM J. Appl. Math., 58 (1998),
  pp.~1748--1773 (electronic).

\bibitem{Figotin:1996a}
{\sc A.~Figotin and P.~Kuchment}, {\em Band-gap structure of spectra of
  periodic dielectric and acoustic media. {I}. {S}calar model}, SIAM J. Appl.
  Math., 56 (1996), pp.~68--88.

\bibitem{Figotin:1996b}
\leavevmode\vrule height 2pt depth -1.6pt width 23pt, {\em Band-gap structure
  of spectra of periodic dielectric and acoustic media. {II}. {T}wo-dimensional
  photonic crystals}, SIAM J. Appl. Math., 56 (1996), pp.~1561--1620.

\bibitem{Fliss:2009}
{\sc S.~Fliss}, {\em Etude math\'{e}matique et num\'{e}rique de la propagation
  des ondes dans des milieux p\'{e}riodiques localement perturb\'{e}s}, PhD
  thesis, Ecole Polytechnique, 5 2009.

\bibitem{Fliss:2008}
{\sc S.~Fliss and P.~Joly}, {\em Exact boundary conditions for time-harmonic
  wave propagation in locally perturbed periodic media}, Appl. Numer. Math.,
  59(9), p.~2155–2178.

\bibitem{Fliss:2011}
\leavevmode\vrule height 2pt depth -1.6pt width 23pt, {\em Wave propagation in
  locally perturbed periodic media (case with absorption): Numerical aspects},
  J. Comput. Phys.,  (2011).

\bibitem{Fliss:2009b}
{\sc S.~Fliss, P.~Joly, and J.-R. Li}, {\em Exact boundary conditions for wave
  propagation in periodic media containing a local perturbation}, Bentham
  Science Publishers, E-Book Series Progress in Computational Physics (PiCP),
  Volume 1, 12 2009.

\bibitem{Hoang:2012}
{\sc V.~Hoang and M.~Radosz}, {\em {Absence of bound states for waveguides in
  2D periodic structures}}, arXiv:1111.4578v1.

\bibitem{Joannopoulos:1995}
{\sc J.~D. Joannopoulos, R.~D. Meade, and J.~N. Winn}, {\em Photonic Crystal -
  Molding the Flow of Light}, Princeton Univeristy Press, 1995.

\bibitem{Johnson:2002}
{\sc S.G. Johnson and J.~D. Joannopoulos}, {\em Photonic Crystal - The road
  from theory to practice}, Kluwer Acad. Publ., 2002.

\bibitem{Fliss:2006}
{\sc P.~Joly, J.-R. Li, and S.~Fliss}, {\em Exact boundary conditions for
  periodic waveguides containing a local perturbation}, Commun. Comput. Phys.,
  1 (2006), pp.~945--973.

\bibitem{Joly:1999}
{\sc P.~Joly and C.~Poirier}, {\em A numerical method for the computation of
  electromagnetic modes in optical fibres}, Math. Methods Appl. Sci., 22
  (1999), pp.~389--447.

\bibitem{Karpeshina:1997}
{\sc Y.~E. Karpeshina}, {\em Perturbation theory for the {S}chr\"odinger
  operator with a periodic potential}, vol.~1663 of Lecture Notes in
  Mathematics, Springer-Verlag, Berlin, 1997.

\bibitem{kato}
{\sc T.~Kato}, {\em Perturbation theory for linear operators}, Classics in
  Mathematics, Springer-Verlag, Berlin, 1995.
\newblock Reprint of the 1980 edition.

\bibitem{Kuchment:1993}
{\sc P.~Kuchment}, {\em Floquet theory for partial differential equations},
  vol.~60 of Operator Theory: Advances and Applications, Birkh\"auser Verlag,
  Basel, 1993.

\bibitem{Kuchment:2001}
\leavevmode\vrule height 2pt depth -1.6pt width 23pt, {\em The mathematics of
  photonic crystals (chapter 7)}, in Mathematical modeling in optical science,
  vol.~22 of Frontiers in applied mathematics, SIAM, Philadelphia, 2001.

\bibitem{Kuchment:2004}
\leavevmode\vrule height 2pt depth -1.6pt width 23pt, {\em On some spectral
  problems of mathematical physics}, in Partial differential equations and
  inverse problems, vol.~362 of Contemp. Math., Amer. Math. Soc., Providence,
  RI, 2004, pp.~241--276.

\bibitem{Kuchment:2003}
{\sc P.~Kuchment and B.~Ong}, {\em On guided waves in photonic crystal
  waveguides}, in Waves in periodic and random media ({S}outh {H}adley, {MA},
  2002), vol.~339 of Contemp. Math., Amer. Math. Soc., Providence, RI, 2003,
  pp.~105--115.

\bibitem{Nazarov:2010}
{\sc S.~A. Nazarov}, {\em {Opening of a gap in the continuous spectrum of a
  periodically perturbed waveguide}}, Mathematical Notes, 87 (2010),
  pp.~738--756.

\bibitem{Nazarov:2011}
{\sc S.~A. Nazarov}, {\em Localized elastic fields in periodic waveguides with
  defects}, J. Appl. Mech. Tech. Phys., 52 (2011), pp.~311--320.

\bibitem{Parnovski:2008}
{\sc L.~Parnovski}, {\em Bethe-{S}ommerfeld conjecture}, Ann. Henri Poincar\'e,
  9 (2008), pp.~457--508.

\bibitem{Parnovski:2010}
{\sc L.~Parnovski and A.~V. Sobolev}, {\em {Bethe-Sommerfeld conjecture for
  periodic operators with strong perturbations}}, Invent. Math., 181 (2010),
  pp.~467--540.

\bibitem{Pedreira:2001a}
{\sc D.~G. Pedreira and P.~Joly}, {\em A method for computing guided waves in
  integrated optics. {I}. {M}athematical analysis}, SIAM J. Numer. Anal., 39
  (2001), pp.~596--623 (electronic).

\bibitem{Pedreira:2001b}
\leavevmode\vrule height 2pt depth -1.6pt width 23pt, {\em A method for
  computing guided waves in integrated optics. {II}. {N}umerical approximation
  and error analysis}, SIAM J. Numer. Anal., 39 (2001/02), pp.~1684--1711
  (electronic).

\bibitem{Popov:1981}
{\sc V.~N. Popov and M.~M. Skriganov}, {\em Remark on the structure of the
  spectrum of a two-dimensional {S}chr\"odinger operator with periodic
  potential}, Zap. Nauchn. Sem. Leningrad. Otdel. Mat. Inst. Steklov. (LOMI),
  109 (1981), pp.~131--133, 181, 183--184.

\bibitem{Reed:1972}
{\sc M.~Reed and B.~Simon}, {\em Methods of modern mathematical physics v.
  I-IV}, Academic Press, New York, 1972-1978.

\bibitem{Sakoda:2001}
{\sc K.~Sakoda}, {\em Optical Properties of Photonic Crystals}, Springer Verlag
  Berlin, 2001.

\bibitem{Schmidt:2010}
{\sc K.~Schmidt and R.~Kappeler}, {\em {Efficient computation of photonic
  crystal waveguide modes with dispersive material}}, Opt. Express, 18 (2010),
  pp.~7307--7322.

\bibitem{Skriganov:1979}
{\sc M.~M. Skriganov}, {\em Proof of the {B}ethe-{S}ommerfeld conjecture in
  dimension {$2$}}, Dokl. Akad. Nauk SSSR, 248 (1979), pp.~39--42.

\bibitem{Skriganov:1983}
\leavevmode\vrule height 2pt depth -1.6pt width 23pt, {\em Multidimensional
  {S}chr\"odinger operator with periodic potential}, Izv. Akad. Nauk SSSR Ser.
  Mat., 47 (1983), pp.~659--687.

\bibitem{Skriganov:1985}
\leavevmode\vrule height 2pt depth -1.6pt width 23pt, {\em The spectrum band
  structure of the three-dimensional {S}chr\"odinger operator with periodic
  potential}, Invent. Math., 80 (1985), pp.~107--121.

\bibitem{Soukoulis:1993}
{\sc C.~M. Soukoulis}, {\em {Photonic band gaps and localization}}, Plenum Pub
  Corp, 1993.

\bibitem{Soukoulis:1996}
\leavevmode\vrule height 2pt depth -1.6pt width 23pt, {\em {Photonic band gap
  materials}}, Springer, 1996.

\bibitem{Soussi:2005}
{\sc S.~Soussi}, {\em Convergence of the supercell method for defect modes
  calculations in photonic crystals}, SIAM J. Numer. Anal., 43 (2005),
  pp.~1175--1201 (electronic).

\bibitem{Vorobets:2011}
{\sc M.~Vorobets}, {\em {On the Bethe-Sommerfeld conjecture for certain
  periodic Maxwell operators}}, J. Math. Anal. Appl., 377 (2011), pp.~370--383.

\bibitem{Yablonovitch:1991}
{\sc E.~Yablonovitch, T.~J. Gmitter, and K.~M. Leung}, {\em Photonic band
  structure: The face-centered-cubic case employing nonspherical atoms}, Phys.
  Rev. Lett., 67 (1991), pp.~2295--2298.

\end{thebibliography}
\end{document}